\documentclass[11pt,a4paper,twoside]{amsart}
\usepackage{amsmath}
\usepackage{amssymb}
\usepackage{amsthm}
\usepackage{mathrsfs}

\allowdisplaybreaks[4]

\usepackage{amsrefs}
\usepackage{hyperref} %%on arXiv

\hypersetup{
    colorlinks=true, %set true if you want colored links
    linktoc= page,   %set to all if you want both sections and subsections linked (elinktoc=nonef (no links) elinktoc=sectionf (default behaviour, same as elinktocpage=falsef) elinktoc=pagef (same as elinktocpage=truef) elinktoc=allf (both the section and page part are links))
}

\numberwithin{equation}{section}

\def \ds{\displaystyle}
\def \wt{\widetilde}

\def \R{\mathbb{R}}
\def \C{\mathbb{C}}
\def \Z{\mathbb{Z}}
\def \N{\mathbb{N}}
\def \v{\varphi}

\def \f{\phi}
\def \g{\psi}
\def \e{\varepsilon}
\def \l{\lambda}

\def \d{\delta}
\def \D{\Delta}

\def \pa{\partial}
\def \n{\nabla}
\def \s{\sigma}
\def \a{\alpha}
\def \b{\beta}

\def \n{\nabla}
\def \t{\theta}

\def \P{\Phi}
\def \ta{\tau}

\def \ga{\gamma}
\def \om{\omega}
\def \ka{\kappa}

\newcommand{\im}{\operatorname{Im}}
\newcommand{\re}{\operatorname{Re}}
\newcommand{\I}{\infty}

\newcommand{\norm}[1]{\left\lVert #1\right\rVert}
\newcommand{\Lebn}[2]{\left\lVert #1 \right\rVert_{L^{#2}}}
\newcommand{\Sobn}[2]{\left\lVert #1 \right\rVert_{H^{#2}}}

\newcommand{\m}[1]{\mathcal{#1}}
\newcommand{\TLebn}[3]{\left\lVert #1 \right\rVert_{L^{#2}(0,T; L^{#3})}}
\newcommand{\TSobn}[3]{\left\lVert #1 \right\rVert_{L^{#2}(0,T; H^{#3})}}

\newcommand{\TFLebn}[4]{\left\lVert #1 \right\rVert_{L^{#2}(0,T; H^{#3}_{#4})}}

\def\({\left(}
\def\){\right)}
\def\<{\left\langle}
\def\>{\right\rangle}
\def\le{\leqslant}
\def\ge{\geqslant}

\theoremstyle{plain}
\newtheorem{theorem}{Theorem}[section]
\newtheorem{definition}[theorem]{Definition}
\newtheorem{lemma}[theorem]{Lemma}

\newtheorem{proposition}[theorem]{Proposition}

\theoremstyle{remark}
\newtheorem{remark}[theorem]{Remark}

\begin{document}
\title[Strong blow-up instability for standing wave to NLKGS]{Strong blow-up instability for standing wave solutions to the system of the quadratic nonlinear Klein-Gordon equations}
\author[H. Miyazaki]{Hayato Miyazaki}
\address[]{Advanced Science Program, Department of Integrated Science and Technology, National Institute of Technology, Tsuyama College, Tsuyama, Okayama, 708-8509, Japan}
\email{miyazaki@tsuyama.kosen-ac.jp}
\date{}

\maketitle

\begin{abstract}
This paper is concerned with strong blow-up instability (Definition \ref{VSI}) for standing wave solutions to the system of the quadratic nonlinear Klein-Gordon equations. In the single case, namely the nonlinear Klein-Gordon equation with power type nonlinearity, stability and instability for standing wave solutions have been extensively studied. On the other hand, in the case of our system, there are no results concerning the stability and instability as far as we know. 

In this paper, we prove strong blow-up instability for the standing wave to our system. The proof is based on the techniques in Ohta and Todorova \cite{OT}. It turns out that we need the mass resonance condition in two or three space dimensions whose cases are the mass-subcritical case.
\end{abstract}

\tableofcontents

\section{Introduction} \label{sec:1}
%\subsection{Background}
In this paper, we consider the system of the quadratic nonlinear Klein-Gordon equation %\footnote{\todayd}
\begin{align}
\begin{cases}
	\ds \frac{1}{c^2} \pa^2_t u -  \D u + m^2 u = \l \overline{u} v, \ (t,x) \in \R_{+} \times \R^{N}, \\[7pt]
	\ds \frac{1}{c^2} \pa^2_t v -  \D v + M^2 v = \mu u^2, \ (t,x) \in \R_{+} \times \R^{N},
\end{cases}
	\label{onlkg} \tag{NLKG}
\end{align}
where $2 \le N \le 5$, $m>0$ and $M >0$ denote the mass of the particles, $\l$ and $\mu$ are complex constants, $c>0$ is the speed of light, and $(u, v)$ is a $\C^2$-valued unknown function with respect to $(t,x)$. 
\eqref{onlkg} is a reduced model of the Dirac-Klein-Gordon system concerning proton-proton interactions and the Maxwell-Higgs system which appears in the abelian Higgs model (cf. \cite{HIN1, YT1}).
Under the mass resonance condition $M = 2m$, \eqref{onlkg} is regarded as a relativistic version of the system of the quadratic nonlinear Schr\"odinger equation
\begin{align}
    \begin{cases}
	\ds i\pa_t v_1 +\frac{1}{2m} \D v_1 = \l_1 \overline{v_1} v_2, \\[7pt]
	\ds i\pa_t v_2 + \frac{1}{2M}\D v_2 = \mu_1 v_1^2.
    \end{cases}
\label{nls} \tag{NLS}
\end{align}
%where $v_j=v_j(t,x):\R^{1+N}$ is an unknown function.
Indeed, by considering the modulated wave solution to \eqref{onlkg} of the form
\[
	(u_c, v_c) = \(e^{itm c^2} u, e^{2itM c^2}v \),
\]
\eqref{onlkg} is rewritten as
\begin{align}
	\begin{cases}
	\ds \frac{1}{c^2} \pa^2_t u_c  - 2im \pa_t u_c - \D u_c  = \l e^{it(2m-M)c^2} \overline{u_c} v_c, \\[7pt]
	\ds \frac{1}{c^2} \pa^2_t v_c  - 2iM \pa_t v_c - \D v_c  = \mu e^{it(M-2m)c^2} u_c^2.
	\end{cases}
	\label{back:1}
\end{align}
Taking $c \to \I$, under $M = 2m$, \eqref{back:1} reaches to \eqref{nls} with $\l_1 = - \l/2m$ and $\mu_1 = - \mu/2M$ (see \cites{HOT}).

We here assume $\l = e \overline{\mu}$ for some $e >0$ and $c=1$ in \eqref{onlkg}. 
By the scaling, 
%\[ 
%	\(u_1(t,x), u_2(t,x)\) = m^{-2} \( \sqrt{e}|\mu|u\(\frac{t}{m}, \frac{x}{m}\), \l v\(\frac{t}{m}, \frac{x}{m}\) \), 
%\]
\eqref{onlkg} can be reduced to
\begin{align}
\begin{cases}
	\pa^2_t u_1 -  \D u_1 + u_1 = \overline{u_1} u_2, \\
	\pa^2_t u_2 -  \D u_2 + \ka^2 u_2 = u_1^2,
\end{cases}
	\label{nlkg} 
\end{align}
where $\ka = M/m$.
We set $\vec{u} = (u_1, u_2)$. 
%In this paper, we are interested in the energy solution $\vec{u} \in C(\R, H^1 \times H^1)$. 
Formally, \eqref{nlkg} has the conserved energy
\begin{align*}
	E(\vec{u}, \pa_t \vec{u}) ={}& \frac12 \( \Lebn{\pa_t u_1}{2}^2 + \frac12\Lebn{\pa_t u_2}{2}^2 \) + \frac12 \( \Lebn{\n u_1}{2}^2 + \frac12 \Lebn{\n u_2}{2}^2 \) \\
	&{}+ \frac12 \( \Lebn{u_1}{2}^2 +\frac{\ka^2}{2} \Lebn{u_2}{2}^2 \) - \frac12 \re \( u_1^2, u_2 \)_{L^2} 
\end{align*}
and the conserved charge
\begin{align*}
	Q(\vec{u}, \pa_t \vec{u}) ={}& \im \int_{\R^N} \overline{u_1}\pa_t u_1 dx + \im \int_{\R^N} \overline{u_2}\pa_t u_2 dx. %\label{charge} 
\end{align*}

The purpose of this paper is to investigate instability of the standing wave solution to \eqref{nlkg} of the form
\[
	\vec{u} = \( e^{i\omega t}\f_{1, \omega}, e^{2i\omega t} \f_{2, \omega} \),
\]
where $\om\in\R$ and $\vec{\f}_{\om} = \(\f_{1, \omega}, \f_{2, \omega}\)$ is a $\R^2$-valued function. 
In the case of \eqref{nls}, Hamano \cite{Ha1} proves strong blow-up instability (see Definition \ref{VSI} below) of standing wave solutions in $N=5$ by giving a threshold for scattering or blow-up below the ground state (see also Dinh \cite{V2}). In \cite{V1}, Dinh investigates stability of standing solutions for $N \le 3$. 
On the other hand, In Garrisi \cite{DG1}, and Zhang, Gan and Guo \cite{ZGG1}, stability of standing wave solutions to the system of nonlinear Klein-Gordon equations has been studied. However the nonlinearity of their system is different from that of \eqref{nlkg}.
In terms of \eqref{nlkg}, stability and instability of standing solutions have not been studied as far as we know.

Let us here deal with known results for stability and instability of the standing wave solution to the focusing nonlinear Klein-Gordon equation with the $p$-th order power nonlinearity
\begin{align}
    \pa_t^2 u - \D u + u = |u|^{p-1}u,
    \label{skg}
\end{align}
where $u=u(t,x)$ is an unknown complex valued function and $p>1$.
In the case of \eqref{skg}, there are many literatures for the stability and instability of standing wave solutions.
Berestycki and Cazenave \cite{BC} (also see Payne and Sattinger \cite{PS1}, and Shatah \cite{Sh2}) prove that the standing wave solution to \eqref{skg} is strongly blow-up unstable when $\om =0$ and $1<p<1+4/(N-2)$, where $\f_{\om}$ is the ground state, namely an unique radially symmetric positive solution to
\begin{align}
	-\D \f + (1-\om^2) \f - |\f|^{p-1} \f = 0. \label{gss:0}
\end{align}
As fundamental works of the orbital stability, Shatah \cite{Sh} prove the orbital stability for standing wave solutions $e^{it \om}\f_{\om}$ to \eqref{skg} if $p<1+4/N$ and $\om_0 < |\om| < 1$, where
\[
	\om_0 = \sqrt{\frac{p-1}{4-(N-1)(p-1)}}.
\]
Shatah and Strauss \cite{ShS} also show that $e^{it \om}\f_{\om}$ is orbitally unstable when $p < 1+ 4/N$ and $|\om| < \om_0$ or $p>1+4/N$ and $|\om| <1$.
We here recall that the orbital stability means stability up to translations and phase shifts as follows: we say that $e^{it \om}\f_{\om}$ is orbital stable if for any $\e >0$, there exists $\d>0$ such that for any $(u_0, u_1) \in H^1(\R^N) \times L^2(\R^N)$ with $\norm{(u_0, u_1) - (\f_{\om}, i\om \f_{\om})}_{H^1 \times L^2} < \d$, the solution $u(t)$ to \eqref{skg} with $u(0) = u_0$ and $\pa_t u(0) = u_1$ exists for all $t \ge 0$ and satisfies 
\[
	\sup_{t \ge 0} \inf_{\t \in \R,\, y \in \R^N} \norm{(u(t), \pa_t u(t)) - e^{i\t}(\f_{\om}(\cdot + y), i\om \f_{\om}(\cdot+y))}_{H^1 \times L^2} < \e.
\]
Otherwise, one calls that $e^{it \om}\f_{\om}$ is orbitally unstable.
As for subsequent results in $N=1$, see Comech and Pelinovsky \cite{CP}, and Wu \cite{Wu}. we also refer to Grillakis, Shatah and Strauss \cite{GSS1, GSS2}, and Maeda \cite{M1} regarding general theory for the orbital stability of solitary wave solutions to an abstract Hamiltonian system.

%An application of the result in Comech and Pelinovsky \cite{CP} leads to the orbital instability when $|\om| = \om_0$, $N=1$ and $p \ge 2$. 
%Wu \cite{Wu} prove the orbital instability when $1<p<1+4/N$, $N=1$, $\om \in (-1, 1)$ and $|\om| = \om_0$.

In what follows, we only handle the case $p=2$, in order to compare \eqref{nlkg} with \eqref{skg}.
Ohta and Torodova \cite{OT2} prove the strong blow-up instability for $e^{it \om}\f_{\om}$ as long as $3 \le N \le 5$ and $|\om| < 1/\sqrt{5}$.
In \cite{OT}, they further establish the strong blow-up instability when $2 \le N \le 5$, $|\om| < 1$ and $|\om| \le \om_0$ if $N =2$, $3$.
In the case $|\om| = \om_0$, it is also shown in \cite{OT} that $e^{it\om} \f$ is strongly blow-up unstable under $N =2$, $3$ and $\f$ is an radially symmetric solution to \eqref{gss:0}.
Liu, Ohta and Todorova \cite{LOT} also establish that $e^{it \om}\f_{\om}$ is the strong blow-up unstable in $N=1$ whenever $0 < \om^2 < 1/6$.

%
%In more general, Grillakis, Shatah and Strauss give stability theory for an abstract Hamiltonian system in \cite{GSS1, GSS2}.
%In terms of the orbital instability, Meada \cite{M1} give a simpler proof, compared with \cite{OT} and \cite{CP}, and unified these proofs as long as $p \ge 2$. 
In view of the previous works, 
as the first step of the investigation for stability and instability of standing wave solutions to the system such as \eqref{nlkg}, we aim to establish instability for the standing wave solution to \eqref{nlkg} in the energy-subcritical cases $2 \le N \le 5$ by applying the technique in \cite{OT}.

\subsection{Main results} 
%\subsection{Definition and Notation}
Let us consider the standing wave solution $\vec{u}$ of the form
\[
	\vec{u} = \( e^{i\omega t}\f_{1, \omega}, e^{2i\omega t} \f_{2, \omega} \)
\]
for all $\om\in\R$, where $\vec{\f}_{\om} = \(\f_{1, \omega}, \f_{2, \omega}\)$ is a $\R^2$-valued function. If $\vec{u}$ satisfies \eqref{nlkg}, then $\vec{\f}_{\om}$ satisfies the system of elliptic equations
\begin{align}
	\begin{cases}
	\ds  -\D \f_{1, \om} + (1 -\om^2) \f_{1, \om} = \f_{1, \om} \f_{2, \om}, \\
	\ds  -\D \f_{2, \om} + (\ka^2- 4\om^2) \f_{2, \om} = \f_{1, \om}^2.
	\end{cases}
	\label{snlkg}
\end{align}
%where $|\om| < \min \(1, \ka/2\)$. 
We here give a definition of solutions to \eqref{snlkg} as follows:

\begin{definition}[\cite{HOT}]
We call that a pair of real-valued functions $\vec{\f} \in (H^1(\R^N))^2$ is a solution to \eqref{snlkg} if 
\begin{align*}
	\int_{\R^N} \n \f_1 \cdot \n v_1\ dx + (1 -\om^2) \int_{\R^N} \f_1 \cdot v_1 \ dx = \int_{\R^N} \f_1 \f_2 \cdot v_1\ dx, \\ 
	\int_{\R^N} \n \f_2 \cdot \n v_2\ dx + (\ka^2 - 4\om^2) \int_{\R^N} \f_2 \cdot v_2 \ dx = \int_{\R^N} \f_1^2 \cdot v_2\ dx 
\end{align*}
for any $\vec{v} \in (C_{0}^\I (\R^N))^2$. Note that $\vec{\f}$ is a solution to \eqref{snlkg} if and only if  $J'(\vec{\f})\vec{v} =0$ for any $\vec{v} \in (H^1(\R^N))^2$.
%$\pa_{u_1} J_{\om}(\vec{\f}) =0$, $\pa_{u_2} J_{\om}(\vec{\f}) =0$.
\end{definition}

The static energy $J_{\omega}$ corresponding to \eqref{snlkg} is defined by
\begin{align*}
	J_{\omega}(\vec{u}) ={}& \frac12 \( \Lebn{\n u_1}{2}^2 + \frac12 \Lebn{\n u_2}{2}^2 \) \\
	&{} + \frac12 \( \(1 - \omega^2\) \Lebn{u_1}{2}^2 + \frac{\ka^2 - 4\omega^2}{2} \Lebn{u_2}{2}^2 \) \\
	&{}- \frac12 \re \( u_1^2, u_2 \)_{L^2}.
%	\label{fun:1}
\end{align*}

The definition of the ground state in this paper is the following:

\begin{definition}[\cite{HOT}] \label{gsdef}
We say that a pair of real-valued functions $\vec{\f}_{\om} \in H^1 \times H^1$ is the ground state for \eqref{snlkg} if $\vec{\f}_{\om} \in \m{G}_{\om}$, where
\begin{align*}
	\m{G}_{\om} ={}& \{\vec{\f}_{\om} \in \m{C}_{\om}\ |\ J_{\om}(\vec{\f}_{\om}) \le J_{\om}(\vec{\v}_{\om})~\text{for any}~\vec{\v}_{\om} \in \m{C}_{\om} \}, \\
	\mathcal{C}_{\om} ={}& \{\vec{\v}_{\om} \in \(H^1_{\rm{rad}}(\R^N)\)^2 \; |\; \vec{\v}_{\om}\; \text{is a nontrivial critical point of}\; J_{\om} \},
\end{align*}
%\begin{align*}
%	&{}J_{\om}(\f_{10}, \f_{20}) = \inf \{ J_{\om}(\vec{\f})\; |\; (\vec{\f}) \in \mathcal{C}_{\om} \}, \\
%	&{}\mathcal{C}_{\om} = \{(\vec{\f}) \in \(H^1(\R^N)\)^2 \; |\; (\vec{\f})\; \text{is a nontrivial critical point of}\; J_{\om} \},
%\end{align*}
and $\vec{\f}$ is called a critical point of $J_{\om}$ if $J'(\vec{\f})\vec{v} =0$ for any $\vec{v} \in (H^1(\R^N))^2$.
\end{definition}

We here remark that it is proved in \cite{HOT} that \eqref{snlkg} has positive radially symmetric solutions in $\m{G}_{\om}$ if $N \le 5$ and $|\om| < \min \(1, \ka/2\)$.

Before we state the main results, we give definitions of strong blow-up instability for the standing wave solution.

\begin{definition}[Strong blow-up instability] \label{VSI}
Let $\vec{\f} \in (H^1(\R^N))^2$. We say $(e^{i\om t}\f_{1}, e^{2i\om t}\f_{2})$ is very strong unstable for \eqref{nlkg} if for any $\e >0$, there exists $(\vec{\v}, \vec{\g}) \in (H^1(\R^N))^2 \times (L^2(\R^N))^2$ such that 
\[
	\norm{(\v_1, \g_1) - (\f_{1}, i\om \f_{1})}_{H^1 \times L^2} + \norm{(\v_2, \g_2) - (\f_{2}, 2i\om \f_{2})}_{H^1 \times L^2} < \e
\]
and the solution $\vec{u}(t)$ to \eqref{nlkg} with $\vec{u}(0) = \vec{\v}$ and $\pa_t \vec{u}(0) = \vec{\g}$ blows up at finite time, namely $T_{\max}< \I$  and
\[
	\lim_{t \to T_{\max-0}} \norm{(\vec{u}, \pa_t \vec{u})(t)}_{H^1 \times L^2} = \I,
\]
where
\[
	\norm{(\vec{u}, \pa_t \vec{u})}_{H^1 \times L^2} = \norm{(u_1, \pa_t u_1)}_{H^1 \times L^2} + \norm{(u_2, \pa_t u_2)}_{H^1 \times L^2}.
\]
\end{definition}

%\begin{definition}[Strong instability] \label{SI}
%Let $\vec{\f} \in H^1 \times H^1$. We say $(e^{i\om t}\f_{1}, e^{2i\om t}\f_{2})$ is very strong unstable for \eqref{nlkg} if for any $\e >0$, there exists $(\vec{\v}, \vec{\g}) \in \m{H}^1 \times \m{H}^0$ such that 
%\[
%	\norm{(\v_1, \g_1) - (\f_{1}, i\om \f_{1})}_{H^1 \times L^2} + \norm{(\v_2, \g_2) - (\f_{2}, 2i\om \f_{2})}_{H^1 \times L^2} < \e
%\]
%and the solution $(u_1(t), u_2(t))$ of \eqref{nlkg} with $\vec{u}_1(0) = (\v_1, \g_1)$ and $\vec{u}_2(0) = (\v_2, \g_2)$ either blows up at finite time, namely $T_{\max}< \I$  and
%\[
%	\lim_{t \to T_{\max}}\norm{(\vec{u}, \pa_t \vec{u})(t)}_{H^1 \times L^2} = \I,
%\]
%or exists globally in time and satisfies 
%\[
%	\limsup_{t \to \I} \norm{(\vec{u}, \pa_t \vec{u})(t)}_{H^1 \times L^2} = \I.
%\]
%\end{definition}

We are in position to state the main results.

%To begin with, we state the main result in this paper.
\begin{theorem} \label{thm:1}
Let $2 \le N \le 5$ and $\om \in \R$ be satisfy $|\om| < \min \(1, \ka/2\)$. Take $\vec{\f}_{\om} \in \m{G}_{\om}$. Assume that $|\om| \le \om_c := \frac{1}{\sqrt{5-N}}$ and $\ka = 2$ if $N=2$, $3$. Then the standing wave solution $(e^{i\om t}\f_{1, \om}, e^{2i\om t}\f_{2, \om})$ to \eqref{nlkg} is strongly blow-up unstable.
\end{theorem}

In the case $|\om| = \om_c$, we have the following:

\begin{theorem} \label{thm:3}
Let $N=2$, $3$. Assume $\ka = 2$. Let $\vec{\f}$ be any nontrivial radially symmetric solution to \eqref{snlkg} with $\om = \om_c$. Then the standing wave solution $(e^{i\om_c t}\f_1, e^{2i\om_c t}\f_2)$ is strongly blow-up unstable. The same assertion occurs in $\om = -\om_c$.
\end{theorem}

%
%The next result is the consequence of Proposition \ref{lem:0} concerning the strong blow-up instability of the standing wave solution.
%
%\begin{corollary} \label{cor:1}
%Let $N =2$, $3$, $4$. Let $\om$ be satisfy $|\om| < \min \(1, \ka/2\)$ satisfying $|\om| \le \om_c$ if $N=2$, $3$. 
%Take $\vec{\f}_{\om} \in \m{G}_{\om}$. 
%Moreover, assume that $\ka = 2$ if $N=2$, $3$. Then the standing wave solution $(e^{i\om t}\f_{1, \om}, e^{2i\om t}\f_{2, \om})$ of \eqref{nlkg} is strongly unstable by blow-up.
%\end{corollary}
%
%In the mass-critical or supercritical case, as in Theorem A in \cite{OT}, the following is valid:
%\begin{theorem} \label{thm:2}
%Let $N=4$, $5$. Set $|\om| < \min \(1, \ka/2\)$. Let $\vec{\f}_{\om} \in \m{G}_{\om}$. Then the standing wave solution $(e^{i\om t}\f_{1, \om}, e^{2i\om t}\f_{2, \om})$ of \eqref{nlkg} is strongly unstable by blow-up.
%\end{theorem}

\begin{remark}
There results are an extension of \cite{OT} regarding \eqref{skg} into \eqref{nlkg}.
\end{remark}

\begin{remark}
Arguing as in \cite{OT} due to \cite{MZ, C1}, the existence of a global solution in time satisfying 
\[
	\limsup_{t \to \I} \norm{(\vec{u}, \pa_t \vec{u})(t)}_{H^1 \times L^2} = \I,
\]
can be excluded.

\end{remark}

\begin{remark}
Compared with \eqref{skg}, we need the mass resonance condition $\ka =2$ when $N =2$, $3$. We do not know if the condition is essential.
\end{remark}
\begin{remark}
In order to apply the technique in \cite{OT}, it is crucial that the ground states are radial. Since \eqref{snlkg} is the special case of the system of elliptic equations considered in Brezis and Lieb \cite{BL1}, all of least energy solutions to \eqref{snlkg} are radially symmetric up to a translation in $\R^N$ if $2 \le N \le 5$ and $|\om| < \min \(1, \ka/2\)$ (see Remark 11 in Byeon, Jeanjean and Mari\c{s} \cite{BJM1}).
%However we are not sure whether all of the ground states, which mean the least energy solutions to \eqref{snlkg}, are radial even if the solutions are positive. 
\end{remark}

\subsection{The key of the proofs}
A strategy of the proofs of the main results rely on the argument in \cite{OT}. This approach is inspired by the technique used to consider some of blow-up problems of solutions to the nonlinear Schr\"odinger equation, for instance, see \cite{OgT1, OgT2}.
In the mass-critical or mass-supercritical case $N=4$, $5$, the key identity of the proof is a local version of a virial type equality
\begin{align}
	-\frac{d}{dt} \sum_{j=1}^2 {}& \frac{1}{j} \re \Big( 2x \cdot \n u_j(t) + Nu_j(t), \pa_t u_j(t) \Big)_{L^2} = K(\vec{u}(t)), \label{virial:1}
\end{align}
where 
\begin{align*}
	K(\vec{u}) ={}& 2\pa_{\l}J_{\om}(\l^{\frac{N}2} \vec{u}(\l \cdot) )|_{\l=1} \\
	={}& 2  \Lebn{\n u_1}{2}^2 + \Lebn{\n u_2}{2}^2  - \frac{N}2 \re \( u_1^2, u_2 \)_{L^2}.
%	\label{fun:2}
\end{align*}
In the mass-subcritical case $N=2$, $3$, the proof is based on a use of a modified virial type equality
\begin{align}
	-\frac{d}{dt} \sum_{j=1}^2 {}& \frac{1}{j} \re \Big( 2x \cdot \n u_j(t) + 4u_j(t), \pa_t u_j(t) \Big)_{L^2} = H(\vec{u}(t), \pa_t \vec{u}(t)), \label{virial:2}
\end{align}
where  $\a = 4-N$ and
\begin{align*}
	H(\vec{u}, \pa_t \vec{u}) ={}& - \a \(\Lebn{\pa_t u_1}{2}^2 + \frac12\Lebn{\pa_t u_2}{2}^2 \) \\
	&{}+ \a \( \Lebn{u_1}{2}^2 +\frac{\ka^2}{2} \Lebn{u_2}{2}^2\) \\
	&{}+(\a+2) \( \Lebn{\n u_1}{2}^2 + \frac12 \Lebn{\n u_2}{2}^2  - \re \( u_1^2, u_2 \)_{L^2} \).
\end{align*}
The left hand side of \eqref{virial:1} and \eqref{virial:2} are not well-defined on $(H^1(\R^N))^2 \times (L^2(\R^N))^2$. Hence we need to approximate the weight function $x$ in \eqref{virial:1} and \eqref{virial:2} by the suitable bounded radial weighted function given in section \ref{sec:3}.
In order to control the error terms generated by the approximation (see Lemma \ref{lem:1} below), one exploits techniques given in the proof of Theorem A in \cite{OT}, which forces to remove the case $N=1$ and the ground states are restricted to be radially symmetric.
%
%the Strauss decay estimate
%\[
%	\norm{f}_{L^{\I}(|x| \ge \rho)} \le C \rho^{-\frac{N-1}{2}} \norm{f}_{H^1_{\rm{rad}}(\R^N)}
%\]
%for any $\rho>0$, 

\subsection{Notation}
We introduce some notations throughout this paper.
For any $p \ge 1$, $L^p(\R^N) = L^p$ denotes the usual Lebesgue space. 
$(1- \D)^{\frac{s}2}$ stands for the usual Bessel potential for all $s \in \R$.
Let $H_p^s(\R^N) = H_{p}^s$ be the Fourier Lebesgue space  $(1- \D)^{-\frac{s}{2}}L^{p}(\R^N)$ for any $p \ge 1$ and all $s \in \R$. We omit the subscript as $H^s$ when $p=2$.
Set $\m{H}^1(\R^N) = \m{H}^1= H^1 \times H^1$ and $\m{H}^0(\R^N) = \m{H}^0 = L^2 \times L^2$.
Let $X_{\rm{rad}} = \{f \in X\ |\ f(x) = f(|x|),\ x \in \R^N \}$ for any Banach spaces $X$.
We often use the following functionals on $\m{H}^1$:
\begin{align*}
	M (\vec{\f}) ={}& \Lebn{\f_1}{2}^2 + \frac{1}2 \Lebn{\f_2}{2}^2, \\
	M_{\om} (\vec{\f}) ={}& (1-\om^2) \Lebn{\f_1}{2}^2 + \frac{\ka^2 - 4\om^2}2 \Lebn{\f_2}{2}^2, \\
	L (\vec{\f}) ={}& \Lebn{\n u_1}{2}^2 + \frac12 \Lebn{\n u_2}{2}^2, \quad L_{\om} (\vec{\f}) = L(\vec{\f}) + M_{\om}(\vec{\f}), \\
	P(\vec{\f}) ={}& \re (\f_1^2, \f_2)_{L^2}, \quad R_{\om}(\vec{\f}) = L_{\om}(\vec{\f})/P(\vec{\f})^{2/3}, \\
	I_{\om}(\vec{\f}) ={}& L(\vec{\f})M_{\om}(\vec{\f})^{1/2}/P(\vec{\f}), \\
	K_{\om}^0(\vec{u}) ={}& \a M_{\om}(\vec{\f}) + (\a +2)(L(\vec{\f}) - P(\vec{\f})),\quad \a= 4-N.
\end{align*}
%for any $\vec{\phi} \in \( H^1(\R^N) \)^2$. 
Then we simply write
\begin{align*}
	J_{\om}(\vec{\f}) ={}& \frac12 L(\vec{\f}) + \frac12 M_{\om}(\vec{\f}) -\frac12 P(\vec{\f}), \quad K(\vec{\f}) = 2 L(\vec{\f}) -\frac{N}2 P(\vec{\f}).
%	K^0_{\om}(\vec{\f}) ={}& \a M_{\om}(\vec{\f}) + (\a +2)(L(\vec{\f}) - P(\vec{\f})).
\end{align*}
The norms for $\C^2$-valued functions is defined by
\begin{align*}
	\norm{\vec{u}}_{L^{p}(I; Y)} ={}& \norm{u_1}_{L^{p}(I; Y)} + \norm{u_2}_{L^{p}(I; Y)}, \\
	\norm{\vec{u}(t)}_{Y} ={}& \norm{u_1(t)}_{Y} + \norm{u_2(t)}_{Y}, \\
	\norm{(\vec{u}, \vec{v})(t)}_{Y \times Z} ={}& \norm{(u_1, v_1)(t)}_{Y \times Z} + \norm{(u_2, v_2)(t)}_{Y \times Z}
\end{align*}
for any $p \ge 1$, all Banach space $Y$, $Z$ and any interval $I \subset \R$, $0 \in I$. 

\medskip

The rest of the paper is organized as follows:
In section \ref{sec:3}, we introduce bounded radial weighted functions necessary to approximate virial type identities \eqref{virial:1} and \eqref{virial:2}, and give key estimates generated by the approximation to show the main results.
Section \ref{sec:4} is devoted to the local well-posedness for \eqref{nlkg}.
In section \ref{sec:5}, we characterize the ground states by using variational argument. 
Some of the lemmas due to the variational argument will be proven in section \ref{vkl}.
We next turn to the proof of the main results in section \ref{sec:6}.
Finally appendix \ref{app:a} provides the proof of the uniform boundedness of solutions to \eqref{nlkg}.

\section{Preliminary} \label{sec:3}
Let us first introduce bounded radial weighted functions which play an important role to prove the main results. Let $\P \in C^{2}([0,\I))$ be a nonnegative function such that
\begin{align*}
	\P(r) =
	\begin{cases}
	N \quad \text{if}\; 0 \le r \le 1, \\
	0 \quad \text{if}\; r \ge 2,
	\end{cases}
	\quad \P'(r) \le 0 \quad \text{if}\; 1 \le r \le 2.
\end{align*}
Set
\begin{align*}
	\P_{\rho}(r) = \P\(\frac{r}{\rho}\), \quad \Psi_{\rho} (r) = \frac{1}{r^{N-1}}\int_0^{r} s^{N-1} \P_{\rho}(s) ds 
\end{align*}
for all $\rho>0$. The following properties are given by \cite{OT}:
\begin{lemma}[{\cite[Lemma 7]{OT}}]
For any $\rho>0$, It holds that
\begin{align}
	&{}\P_{\rho}(r) = N,\; \Psi_{\rho}(r) = r \; \text{if}\; 0 \le r \le \rho, \quad \Psi_{\rho}(r) \le 2^N \rho \; \text{if}\; r \ge \rho, \label{we:2} \\
	&{}\Psi'_{\rho}(r) + \frac{N-1}{r}\Psi_{\rho}(r) = \P_{\rho}(r), \quad r \ge 0, \label{we:3} \\
	&{}|\P_{\rho}^{(k)}(r)| \le \frac{C}{\rho^k}, \quad r \ge 0,\; k=0, 1, 2, \label{we:4} \\
	&{}\Psi'_{\rho}(r) \le 1, \quad r \ge0. \label{we:5}
\end{align}
\end{lemma}

Here is key estimates due to approximating virial type identities \eqref{virial:1} and \eqref{virial:2} to show the main results.
\begin{lemma} \label{lem:1}
Let $\vec{u} \in C([0,T_{\max}), \m{H}^1)$ be a radially symmetric maximal solution of \eqref{nlkg}. Then there exists $C_0>0$ such that
\begin{align}
	\begin{aligned}
	-\frac{d}{dt}I_{\rho}^1(\vec{u}(t), \pa_t \vec{u}(t)) \le{}& K(\vec{u}(t)) + \frac12 \re \int_{\{|x| \ge \rho\}}u_1(t)^2 \overline{u_2(t)} dx \\
	&{}+ \frac{C_0}{\rho^2} M(\vec{u}(t)),
	\end{aligned}
	\label{ineq:1} \\
	\begin{aligned}
	-\frac{d}{dt}I_{\rho}^2(\vec{u}(t), \pa_t \vec{u}(t)) \le{}& H(\vec{u}(t), \pa_t \vec{u}(t)) + \frac12 \re \int_{\{|x| \ge \rho\}}u_1(t)^2 \overline{u_2(t)} dx \\
	&{}+ \frac{C_0}{\rho^2} M(\vec{u}(t))
	\end{aligned}
	\label{ineq:2}
\end{align}
for any $t \in [0,T_{\max})$ and all $\rho>0$, where
\begin{align}
	I_{\rho}^1(\vec{u}, \pa_t \vec{u}) ={}& \sum_{j=1}^2 \frac1{j} \re \( 2\Psi_{\rho} \pa_r u_j + \P_{\rho} u_j, \pa_t u_j\)_{L^2}, \label{qu:1} \\
	I_{\rho}^2 (\vec{u}, \pa_t \vec{u}) ={}& \sum_{j=1}^2 \frac1{j} \re \( 2\Psi_{\rho} \pa_r u_j + \P_{\rho} u_j + \a u_j, \pa_t u_j\)_{L^2}.
	\label{qu:2}
\end{align}

\end{lemma}
\begin{proof}
Since $\vec{u}$ is a radially symmetric with respect to $x$, we remark that
\begin{align*}
	&{}\D u_j = \pa_r^2 u_j + \frac{N-1}{r} \pa_r u_j, \quad \n \cdot \(\frac{x}{r} \) = \frac{N-1}{r}, \\
	&{}\pa_r u_j = \frac{x \cdot \n}{r} u_j, \quad |\pa_r u_j| = |\n u_j|
\end{align*}
for $j=1$, $2$. By using the integration by part and \eqref{nlkg}, we have
\begin{align*}
	&{}-\frac{d}{dt} \sum_{j =1}^2 \frac1{j} \re \( 2\Psi_{\rho} \pa_r u_j, \pa_t u_j\)_{L^2} \\
	={}& \int_{\R^N} \( \Psi'_\rho + \frac{N-1}{r} \Psi_{\rho} \) \(|\pa_t u_1|^2 + \frac12|\pa_t u_2|^2  \) dx \\
	&{} + \int_{\R^N} \( \Psi'_\rho - \frac{N-1}{r} \Psi_{\rho} \) \(|\n u_1|^2 + \frac12|\n u_2|^2  \) dx \\
	&{} + \int_{\R^N} \( \Psi'_\rho + \frac{N-1}{r} \Psi_{\rho} \) \( |u_1|^2 + \frac{\ka^2}2 |u_2|^2  \) dx \\
	&{} + \re \int_{\R^N} \( \Psi'_\rho + \frac{N-1}{r} \Psi_{\rho} \)u_1^2 \overline{u_2} dx.
\end{align*}
Arguing as in the above, one also obtains
\begin{align*}
	&{}-\frac{d}{dt} \sum_{j=1}^2 \frac1{j} \re \( \P_{\rho} u_j, \pa_t u_j\)_{L^2} \\
	={}& - \int_{\R^N} \P_{\rho} \(|\pa_t u_1|^2 + \frac12|\pa_t u_2|^2  \) dx - \frac12 \int_{\R^N} \D \P_{\rho} \(|u_1|^2 + \frac12|u_2|^2  \) dx \\
	&{} + \int_{\R^N} \P_{\rho} \(|\n u_1|^2 + \frac12|\n u_2|^2  \) dx + \int_{\R^N} \P_{\rho} \( |u_1|^2 + \frac{\ka^2}2 |u_2|^2  \) dx \\
	&{} - \frac32 \re \int_{\R^N} \P_{\rho} u_1^2 \overline{u_2} dx.
\end{align*}
Combining these identities with \eqref{we:2}, \eqref{we:3}, \eqref{we:4} and \eqref{we:5}, we have \eqref{ineq:1}. Moreover, since a calculation shows
\begin{align*}
	&{}-\frac{d}{dt} \sum_{j=1}^2 \frac1{j} \re (u_j, \pa_t u_j)_{L^2} \\
	={}& - M(\pa_t \vec{u}) + L(\vec{u}) + \( \Lebn{u_1}{2}^2 + \frac{\ka^2}{2} \Lebn{u_2}{2}^2 \) - \frac32 P(\vec{u}), 
\end{align*}
one reaches to \eqref{ineq:2}.
\end{proof}

\section{Local well-posedness in the energy space} \label{sec:4}

In this section, by arguing as in \cite{BW1}, we shall prove the local well-posedness in the energy space $\m{H}^1 \times \m{H}^0$  for the Cauchy problem  
\begin{align}
	\begin{cases}
	\pa^2_t u_1 -  \D u_1 + u_1 = \overline{u_1} u_2,\; (t, x) \in \R_{+} \times \R^N \\
	\pa^2_t u_2 -  \D u_2 + \ka^2 u_2 = u_1^2,\; (t, x) \in \R_{+} \times \R^N \\
	\vec{u}(0) = \vec{u}^0,\; \pa_t \vec{u}(0) = \vec{u}^1,\; x \in \R^N,
	\end{cases}
	\label{cau:1}
\end{align}
where $\vec{u} = (u_1, u_2)^t$ is a $\C^2$-valued unknown function, $\vec{u}^0 = (u_1^0, u_2^0)^t$, $\vec{u}^1 = (u_1^1, u_2^1)^t$ and $\ka >0$.
We here define the following notations:
\begin{align*}
	K_1(t) ={}& \frac{\sin (t\sqrt{1-\D})}{\sqrt{1-\D}}, \quad K_2(t) = \frac{\sin (t\sqrt{\ka^2-\D})}{\sqrt{\ka^2-\D}}, \\
	U(t) ={}& 
	\begin{pmatrix}
	K_1(t) & 0 \\
	0 &  K_2(t)
	\end{pmatrix}
	, \quad F(\vec{u}) =
	\begin{pmatrix}
	\overline{u_1} u_2 \\
	u_1^2
	\end{pmatrix}
	.
\end{align*}
We give the definition of solutions to \eqref{cau:1}. 
\begin{definition}[Solution]\label{def:sol}
We say a function $\vec{u}(t)$ is a solution to \eqref{cau:1} on an interval $I \subset \R$, $I \ni 0$
if $\vec{u} \in C(I, \m{H}^1) \cap C^1(I, \m{H}^0)$ and satisfies
\begin{align*}
	\vec{u}(t) = \dot{U}(t)\vec{u}^0 + U(t)\vec{u}^1 + \int_0^t U(t-\ta) F(\vec{u}(\ta)) d\ta
\end{align*}
in $\m{H}^1$ for any $t \in I$.
We call $I$ is a maximal interval of $\vec{u}$ if $\vec{u}(t)$ cannot be
extended to any interval strictly larger than $I$.
We denote the maximal interval of $\vec{u}$ by $I_{\max} = I_{\max}(\vec{u}) = (T_{\min}, T_{\max})$.
\end{definition}

We have the following:

%We state local well-posedness for (\ref{nlkg}) in the energy space $H^1 \times L^2^2$ in the enery-critical or enery-subcritical case $N\in \{1,\cdots,6\}$.
\begin{proposition}[Local well-posedness for \eqref{cau:1} in $\m{H}^1 \times \m{H}^0$] \label{lwp:1}
Assume that $\ka >0$. Let $2 \le N \le 5$. 
%$\vec{u}^0 \in \m{H}^1$ and $\vec{u}^1 \in \m{H}^0$. 
Then the Cauchy problem \eqref{cau:1} is locally well-posed in $\m{H}^1 \times \m{H}^0$. Namely, for any $(\vec{u}^0, \vec{u}^1) \in \m{H}^1 \times \m{H}^0$, there exists a $T>0$ and an unique solution $\vec{u} \in C([0, T], \m{H}^1)$ to \eqref{cau:1} satisfying $\pa_t \vec{u} \in C([0,T], \m{H}^0)$. Furthermore, the map $(\vec{u}^0,  \vec{u}^1) \mapsto u$ is continuous in the following sense:
For any compact $I' \subset [0, T]$, there exists a neighborhood $V$ of $(\vec{u}^0, \vec{u}^1)$ in $\m{H}^1 \times \m{H}^0$ such that the map 
is Lipschitz continuous from $V$ to $C(I', \m{H}^1)$.
\end{proposition}
In $N=5$, the key of the proof of Proposition \ref{lwp:1} is the Strichartz estimate as follows:

\begin{proposition}[\cite{P, BW1}] \label{str:0}
Denote $V(t) = e^{it(m^2 -\D)^{1/2}}$ for any $m>0$. Let $I$ be an any interval in $\R$.
Assume $r$, $q \in [2, \I)$ if $N=2$, $3$ and $r$, $q \in [2, 2(N-1)/(N-3))$ if $N \ge 4$. Let $\s(r)$, $\ga(r)$ and $\s(q)$, $\ga(q)$ satisfy 
\begin{align*}
	\frac{2\s(r)}{N+1} = \frac{2}{\ga(r)(N-1)} = \frac12 - \frac1{r}. %\label{str:1}
\end{align*}
Then the following estimates hold:
\begin{align*}
	\norm{V(t)\v}_{L^{\ga(r)}\(I; H^{s-\s(r)}_{r}\)} \le{}& C \Sobn{\v}{s}, \\ %\label{str:2}
	\norm{\int_{0}^t V(t-\ta)f(\ta)\ d\ta}_{L^{\ga(q)}\(I; H^{s-\s(q)}_{q}\)} \le{}& C \norm{f}_{L^{\ga(r)'}\(I; H^{s+\s(r)}_{r'}\)}, %\label{str:3}
\end{align*}
where $p'$ is defined by $\frac1{p} + \frac1{p'} =1$ for any $p \ge 2$.
\end{proposition}

\begin{proof}[Proof of Proposition \ref{lwp:1}]
%We only consider forward direction in time, because backward case is similar. 
Let us introduce the complete metric space
%\begin{align*}
%X_{T, M} ={}& \left\{ \left. \vec{u} \in \( L^{\I}(0, T; H^1) \cap L^{\ga(\rho)}(0, T; H^{1-\s(\rho)}_{\rho})\)^2\ \right|  \right. \\
%	&{} \qquad \left. \norm{\vec{u}}_{L^{\ga(\rho)}(0, T; H^{1-\s(\rho)}_{\rho})} \le M \right\}
%\end{align*}
\begin{align*}
X_{T, M} ={}& \left\{ \left. \vec{u} \in \( L^{\ga(\rho)}(0, T; H^{1-\s(\rho)}_{\rho})\)^2\ \right| \norm{\vec{u}}_{L^{\ga(\rho)}(0, T; H^{1-\s(\rho)}_{\rho})} \le M \right\}
\end{align*}
equipped with the distance function
\begin{align*}
	d_{T, M}(\vec{u}, \vec{v}) ={}& \norm{\vec{u} - \vec{v}}_{L^{\ga(\rho)}(0, T; L^{\rho})}, 
\end{align*}
where $(\rho, \ga, \s) = (2, \I, 0)$ if $N=2$, $3$, $4$, $(\rho, \ga, \s) = (12/5, 6, 1/4)$ if $N=5$.
%\[
%	\norm{\vec{u}}_{L^{\ga(\rho)}(0, T; H^{1-\s(\rho)}_{\rho})} 
%	= \norm{u_1}_{L^{\ga(\rho)}(0, T; H^{1-\s(\rho)}_{\rho})} + \norm{u_2}_{L^{\ga(\rho)}(0, T; H^{1-\s(\rho)}_{\rho})}.
%\]
Here the constant $M$ will be chosen later. Set
\[
	\P(\vec{u}) = \dot{U}(t)\vec{u}^0 + U(t)\vec{u}^1 + \int_0^t U(t-\ta) F(\vec{u}(\ta)) d\ta.
\]
We shall prove that $\P$ is a contraction map in $X_{T, M}$. Let us first show that $\P$ maps from $X_{T, M}$ to itself.
When $N=2$, $3$, $4$, we estimate
\begin{align*}
	\TLebn{\P(\vec{u})}{\I}{2} \le{}& \TLebn{\dot{U}(t)\vec{u}^0}{\I}{2} + \TLebn{U(t)\vec{u}^1}{\I}{2} \\
	&{}+ \TLebn{\int_0^t U(t-\ta) F(\vec{u}(\ta)) d\ta}{\I}{2} \\
	\le{}& \Lebn{\vec{u}^0}{2} + \Lebn{(1-\D)^{-\frac12} \vec{u}^1}{2} + \TLebn{(1-\D)^{-\frac12} F(\vec{u})}{1}{2} \\
	\le{}& \Sobn{\vec{u}^0}{1} + \Lebn{\vec{u}^1}{2} + \TSobn{F(\vec{u})}{1}{-1}.
\end{align*}
By Sobolev embedding, the last term can be calculated as follows:
\begin{align*}
	\TSobn{F(\vec{u})}{1}{-1} \le{}& T \TLebn{F(\vec{u})}{\I}{2} \\
	\le{}& T\( \TLebn{u_1}{\I}{4}\TLebn{u_2}{\I}{4} + \TLebn{u_1}{\I}{4}^2 \) \\
	\le{}& CT \TSobn{\vec{u}}{\I}{1}^2.
\end{align*}
Hence we see that
\begin{align*}
	\TLebn{\P(\vec{u})}{\I}{2} \le \Sobn{\vec{u}^0}{1} + \Lebn{\vec{u}^1}{2} + CT M^2.
\end{align*}
In similar way, it follows from Sobolev embedding that 
\begin{align*}
	\TLebn{\n \P(\vec{u})}{\I}{2} \le{}& \TLebn{\dot{U}(t) \n \vec{u}^0}{\I}{2} + \TLebn{U(t) \n \vec{u}^1}{\I}{2} \\
	&{}+ \TLebn{\int_0^t U(t-\ta) \n F(\vec{u}(\ta)) d\ta}{\I}{2} \\
	\le{}& \Sobn{\vec{u}^0}{1} + \Lebn{\vec{u}^1}{2} + \TLebn{F(\vec{u})}{1}{2} \\
	\le{}& \Sobn{\vec{u}^0}{1} + \Lebn{\vec{u}^1}{2} + CT M^2.
\end{align*}
Combining these above, we see that
\begin{align*}
	\TSobn{\P(\vec{u})}{\I}{1} \le 2\( \Sobn{\vec{u}^0}{1} + \Lebn{\vec{u}^1}{2}\) + CT M^2. %\label{cau:3}
\end{align*}
Putting $M = 4\( \Sobn{\vec{u}^0}{1} + \Lebn{\vec{u}^1}{2}\)$, $\TSobn{\P(\vec{u})}{\I}{1} \le M$ holds as long as $T = T \( \Sobn{\vec{u}^0}{1} + \Lebn{\vec{u}^1}{2} \)$ satisfies
\begin{align}
	CT M^2 \le \frac{M}{2}. \label{cau:5}
\end{align}
Thus $\P(\vec{u}) \in X_{T, M}$. We shall prove $\P$ is a contraction map in $X_{T, M}$.
Note that
\begin{align}
	|F(\vec{u}) - F(\vec{v})| \le C\(|u_1| + |u_2| + |v_1| \)\(|u_1 -v_1| + |u_2 -v_2| \) \label{cau:6}
\end{align}
for any $\vec{u}$, $\vec{v} \in X_{T,M}$. 
We calculate
\begin{align*}
	\TLebn{\P(\vec{u}) - \P(\vec{u})}{\I}{2} \le{}& \TLebn{\int_0^t U(t-\ta) \( F(\vec{u}(\ta)) - F(\vec{v}(\ta)) \) d\ta}{\I}{2} \\
	\le{}& \TSobn{F(\vec{u}) - F(\vec{v})}{1}{-1+\eta},
\end{align*}
where $\eta = (4-N)/4$. Let
\[
	\frac1{\l_1} = \frac12 - \frac{-1+\eta}{N} = \frac12 + \frac14.
\]
Combining Sobolev embedding with \eqref{cau:6}, one obtains 
\begin{align*}
	&{}\Sobn{F(\vec{u}) - F(\vec{v})}{-1+\eta} \\
	\le{}& C\Lebn{F(\vec{u}) - F(\vec{v})}{\l_1} \\
	\le{}& C\(\Lebn{u_1}{4} + \Lebn{u_2}{4} + \Lebn{v_1}{4} \)\(\Lebn{u_1-v_1}{2} + \Lebn{u_2 - v_2}{2} \) \\
	\le{}& C\(\Sobn{u_1}{1} + \Sobn{u_2}{1} + \Sobn{v_1}{1} \)\(\Lebn{u_1-v_1}{2} + \Lebn{u_2 - v_2}{2} \),  
\end{align*}
which yields
\begin{align*}
	\TLebn{\P(\vec{u}) - \P(\vec{u})}{\I}{2} \le CT M \TLebn{\vec{u}- \vec{v}}{\I}{2}. 
\end{align*}
Hence $\P$ is a contraction map in $X_{T,M}$ whenever $T = T \( \Sobn{\vec{u}^0}{1} + \Lebn{\vec{u}^1}{2} \)$ satisfies 
\begin{align}
	CTM \le 1/2. \label{cau:7}
\end{align}
Therefore, taking $T = T_0 \( \Sobn{\vec{u}^0}{1} + \Lebn{\vec{u}^1}{2} \)$ satisfying \eqref{cau:5} and \eqref{cau:7}, we have a solution to \eqref{cau:1} in $X_{T_0, M}$.

Let us move on to the case $N=5$. 
We deduce from Proposition \ref{str:0} that 
\begin{align*}
	\TFLebn{\P(\vec{u})}{\ga}{1-\s}{\rho} \le{}& \TFLebn{\dot{U}(t)\vec{u}^0}{\ga}{1-\s}{\rho} + \TFLebn{U(t)\vec{u}^1}{\ga}{1-\s}{\rho} \\
	&{}+ \TFLebn{\int_0^t U(t-\ta) F(\vec{u}(\ta)) d\ta}{\ga}{1-\s}{\rho} \\
	\le{}& C \Sobn{\vec{u}^0}{1} + C\Sobn{(1-\D)^{-\frac12} \vec{u}^1}{1} \\
	&{}+ C \TFLebn{(1-\D)^{-\frac12} F(\vec{u})}{\ga'}{1+\s}{\rho'} \\
	\le{}& C\( \Sobn{\vec{u}^0}{1} + \Lebn{\vec{u}^1}{2} \) + C \TFLebn{F(\vec{u})}{\ga'}{\s}{\rho'}.
\end{align*}
%We here remark that 
%\begin{align}
%	B^{s}_{p, \min(p,2)} \hookrightarrow H^{s, p} \hookrightarrow B^{s}_{p, \max(p,2)} \label{cau:9}
%\end{align}
%for any $s \in \R$ and all $p \ge 1$. 
Set 
\[
	\frac{1}{\l_2} = \frac1{\rho'} - \frac{\s}{5}, \quad \frac1{\a_0} = \frac1{\rho} - \frac{1-\s}{5}.
\]
Because of $1/\l_2 = 2/\a_0$, one has
\begin{align*}
%	\norm{uv}_{B^{\s}_{\rho', 2}} \le{}& C 
	\norm{uv}_{H^{\s, \rho'}} \le{}& C\Lebn{uv}{\l_2} \le C\norm{u}_{H^{1-\s, \rho}} \norm{v}_{H^{1-\s, \rho}} \\
%	\le{}& C\norm{u}_{B^{1-\s}_{\rho, 2}} \norm{v}_{B^{1-\s}_{\rho, 2}}. 
\end{align*}
This implies that
\begin{align}
	\begin{aligned}
	\TFLebn{F(\vec{u})}{\ga'}{\s}{\rho'} ={}& \TFLebn{u_1 u_2}{\ga'}{\s}{\rho'} + \TFLebn{u_1^2}{\ga'}{\s}{\rho'} \\
	\le{}& CT^{1/2} \( \TFLebn{u_1}{\ga}{1-\s}{\rho} + \TFLebn{u_2}{\ga}{1-\s}{\rho} \)^2
	\end{aligned}
	\label{cau:11}
\end{align}
from which we conclude that
\begin{align}
	\TFLebn{\P(\vec{u})}{\ga}{1-\s}{\rho} \le C\( \Sobn{\vec{u}^0}{1} + \Lebn{\vec{u}^1}{2} \) +CT^{1/2} M^2. \label{cau:4}
\end{align}
Putting $M = 2C\( \Sobn{\vec{u}^0}{1} + \Lebn{\vec{u}^1}{2}\)$, $\TSobn{\P(\vec{u})}{\I}{1} \le M$ holds as long as $T = T \( \Sobn{\vec{u}^0}{1} + \Lebn{\vec{u}^1}{2} \)$ satisfies
\begin{align}
	CT^{1/2} M^2 \le \frac{M}{2}. \label{cau:8}
\end{align}
Thus $\P(\vec{u}) \in X_{T, M}$. We shall prove that $\P$ is a contraction map in $X_{T, M}$.
It come from Proposition \ref{str:0} that
\begin{align*}
	&{}\TLebn{\P(\vec{u}) - \P(\vec{v})}{\ga}{\rho} \\
	\le{}& \TLebn{\int_0^t U(t-\ta) \( F(\vec{u}(\ta)) - F(\vec{v}(\ta)) \) d\ta}{\ga}{\rho} \\
	\le{}& C \TFLebn{F(\vec{u}) - F(\vec{v})}{\ga'}{2\s-1}{\rho'}.
\end{align*}
Put 
\begin{align*}
	\frac{1}{\l_3} ={}& \frac1{\rho'} - \frac{2\s-1}{5}, \quad \frac1{a_0} = \frac1{\rho} - \frac{1-\s}{5}, \quad \frac1{\a_1} = \frac1{\rho}, \\
	\frac1{\l_3} ={}& \frac1{\a_0} + \frac1{\a_1}.
\end{align*}
Hence we see from \eqref{cau:6} that
\begin{align*}
	&{}\TFLebn{F(\vec{u}) - F(\vec{v})}{\ga'}{2\s-1}{\rho'} \\
	\le{}& C \TLebn{F(\vec{u}) - F(\vec{v})}{\ga'}{\l_3} \\
	\le{}& CT^{1/2}\(\TFLebn{u_1}{\ga}{1-\s}{\rho} + \TFLebn{u_2}{\ga}{1-\s}{\rho} + \TFLebn{v_1}{\ga}{1-\s}{\rho} \) \\
	&{} \times \TLebn{\vec{u} -\vec{v}}{\ga}{\rho},
\end{align*}
which implies
\begin{align*}
	\TLebn{\P(\vec{u}) - \P(\vec{v})}{\ga}{\rho} \le CT^{1/2} M \TLebn{\vec{u} -\vec{v}}{\ga}{\rho}
\end{align*}
Hence $\P$ is a contraction map in $X_{T,M}$ as long as $T = T \( \Sobn{\vec{u}^0}{1} + \Lebn{\vec{u}^1}{2} \)$ satisfies 
\begin{align}
	CT^{1/2} M \le 1/2. \label{cau:10}
\end{align}
Therefore, taking $T = T \( \Sobn{\vec{u}^0}{1} + \Lebn{\vec{u}^1}{2} \)$ satisfying \eqref{cau:8} and \eqref{cau:10}, we have a solution to \eqref{cau:1} in $X_{T, M}$. Computing as in \eqref{cau:4}, we also see from Proposition \ref{str:0} that $\vec{u} \in \( L^{\I}(0,T; H^1)\)^2$. Finally we show $\pa_t \vec{u} \in L^{\I}(0, T; L^2)$. Since we have
\begin{align*}
	\pa_t \vec{u} = \ddot{U}(t) \vec{u}^0 + \dot{U}(t)\vec{u}^1 + \int_0^t \dot{U}(t-\ta) F(\vec{u})(\ta) d\ta,
\end{align*}
it is established that
\begin{align*}
	\TLebn{\pa_t \vec{u}}{\I}{2} \le{}& C\( \Sobn{\vec{u}^0}{1} + \Lebn{\vec{u}^0}{2} \) \\
	&{} + \TLebn{\int_0^t \dot{U}(t-\ta) F(\vec{u})(\ta) d\ta}{\I}{2}.
\end{align*}
When $2 \le N \le 4$, we easily see from Sobolev embedding that
\begin{align*}
	\TLebn{\int_0^t \dot{U}(t-\ta) F(\vec{u})(\ta) d\ta}{\I}{2} \le{}& \TLebn{F(\vec{u})}{1}{2} \\
	\le{}& T \TSobn{\vec{u}}{\I}{1}.
\end{align*}
In $N=5$, arguing as in \eqref{cau:11}, making a use of Proposition \ref{str:0}, one has 
\begin{align*}
	\TLebn{\int_0^t \dot{U}(t-\ta) F(\vec{u})(\ta) d\ta}{\I}{2} \le{}& C \TFLebn{F(\vec{u})}{\ga'}{\s}{\rho'} \\
	\le{}& CT^{1/2} \TFLebn{\vec{u}}{\ga}{1-\s}{\rho}^2.
\end{align*}
Therefore it is concluded that $\pa_t \vec{u} \in L^{\I}(0, T; L^2)$. The remainder of the proof is standard, so we omit the proof. 
\end{proof}

\section{Characterization of the ground states} \label{sec:5}

%\subsection{Characterization of the ground state}

In this section, we will prove the existence of the ground states characterized by the solution to two constrained minimization problems
%By the scaling 
%\[
%	\psi_{j, \om} (x) = \frac{1}{1-\om^2} \f_{j, \om} \(\frac{x}{\sqrt{1-\om^2}} \), \quad j=1, 2,
%\]
%\eqref{snlkg} can be rewritten as
%\begin{align}
%	\begin{cases}
%	\ds  -\D \g_{1, \om} + \g_{1, \om} = \g_{1, \om} \g_{2, \om}, \\
%	\ds  -\D \g_{2, \om} + \ka_{\om} \g_{2, \om} = \g_{1, \om}^2,
%	\end{cases}
%	\label{snlkg1}
%\end{align}
%where $\ka_{\om} = (\ka^2- 4\om^2)/(1 -\om^2)$.
%We here consider the two constrained minimization problems
\begin{align}
d_{\om}^1 := \inf \left\{ J_{\om}(\vec{\phi})\ \left|\ 
	\vec{\phi} \in (H^1(\R^{N}, \R))^2 \setminus \vec{0},\; K(\vec{\phi}) = 0
	\right.  \right\} \label{mini:1}
\end{align}
in $N=4$, $5$, and
\begin{align}
d_{\om}^0 := \inf \left\{ J_{\om}(\vec{\phi})\ \left|\ 
	\vec{\phi} \in (H^1(\R^{N}, \R))^2 \setminus \vec{0},\; K_{\om}^0(\vec{\phi}) = 0
	\right.  \right\} \label{mini:2}
\end{align}
in $N=2$, $3$, where
\begin{align*}
	K_{\om}^0(\vec{u}) 
	={}& 2\pa_{\l}J_{\om}(\l^{2} \vec{u}(\l \cdot) )|_{\l=1} \\
%	={}& \a \( (1- \omega^2) \Lebn{u_1}{2}^2 + \frac{\ka_{\om}^2 - 4\omega^2}{2} \Lebn{u_2}{2}^2 \) \\
%	&{} + (\a+2) \( \Lebn{\n u_1}{2}^2 + \frac12 \Lebn{\n u_2}{2}^2 - \re \( u_1^2, u_2 \)_{L^2} \) \\
	={}& \a M_{\om}(\vec{\f}) + (\a +2)(L(\vec{\f}) - P(\vec{\f})).
\end{align*}
%We will prove later that there exists a solution $\vec{\f}$ to \eqref{snlkg} such that it attains the above minimization problem.
We further set 
\begin{align*}
	\m{K} ={}& \{\vec{\f} \in (H^1(\R^N, \R))^2 \setminus \vec{0}\ |\ K(\vec{\f}) =0 \}, \\
	\m{K}_{\om}^0 ={}& \{\vec{\f} \in (H^1(\R^N, \R))^2 \setminus \vec{0}\ |\ K_{\om}^0(\vec{\f}) =0 \}, \\
\m{M}_{\om} ={}& \left\{ \vec{\f} \in \( H^1_{\rm{rad}}(\R^N, \R) \)^2 \setminus \vec{0}\ \left|\ 
	\begin{array}{c}
	J_{\om}( \vec{\phi}) = d_{\om}^1, \; \vec{\f} \in \m{K}\ (N=4, 5) \\ 
	J_{\om}( \vec{\phi}) = d_{\om}^0, \; \vec{\f} \in \m{K}_{\om}^0 \ (N=2, 3)
	\end{array}
	\right. \right\}.
\end{align*}
%Let us introduce the notations. Set 
%\begin{align*}
%A ={}& \left\{ \vec{\f} \in (H^1(\R^N))^2 ; \vec{\f}~\text{is a solution of \eqref{snlkg}} \right\}, \\
%\m{G}_{\om} ={}& \left\{ \vec{\f} \in A ; J_{\om}(\vec{\f}) \le J_{\om}(\vec{v})~\text{for any}~\vec{v} \in A \right\}.
%\end{align*}

%Further, the static energy is defined by 
%\begin{align*}
%	J(u, v) = \sum_{j=1,2} \frac1{2j} \( \Lebn{\n u_j}{2}^2 + m^2_j \Lebn{u_j}{2}^2  \)
%-\frac12 \int_{\R^5} u^2 v dx.
%\end{align*}
%Also, we set 
%\begin{align*}
%	I(u, v) &= \left. \frac{d}{d\l} J(\l u, \l v) \right|_{\l=1} \\
%	&= \sum_{j=1,2} \frac{1}{j} \( \Lebn{\n u_j}{2}^2 + m_j^2 \Lebn{u_j}{2}^2 \) - \frac32 \int_{\R^{5}}u^2 v dx. 
%\end{align*}

The following holds:

\begin{proposition} \label{gs:t0}
Let $2 \le N \le 5$. Then $\m{G}_{\om} = \m{M}_{\om}$. 
\end{proposition}

The next lemma is very helpful to show Proposition \ref{gs:t0}.
\begin{lemma}[{\cite[Theorem 4.1]{HOT}}] \label{lgs:0}
Let $\vec{\f} \in \m{H}^1$ be a solution to \eqref{snlkg}. Then following hold:
\begin{enumerate}
\renewcommand{\labelenumi}{(\roman{enumi})}
\item (Pohozaev identity) 
\begin{align*}
	\left. \frac{d}{d \l}J_{\om}\(\vec{\f}(\cdot/\l)\) \right|_{\l=1} =0 \Leftrightarrow (N-2) L(\vec{\f})+ N M_{\om}(\vec{\f}) = N P(\vec{\f}),
\end{align*} 
	
\item $\ds 2 L(\vec{\f}) + 2 M_{\om}(\vec{\f}) = 3 P(\vec{\f})$, 
\item $\ds N M_{\om}(\vec{\f}) = (6-N) L(\vec{\f})$, 
\item $\ds N K_{\om}^0 (\vec{\f}) = 2(6-N)K(\vec{\f}) = 0$. 
\end{enumerate}

\end{lemma}

We divide the proof of Proposition \ref{gs:t0} into the mass-supercritical case $N=5$, the mass-critical case $N=4$, and the mass-subcritical case $N=2$, $3$. 
Let us begin with the case $N=5$. 
In order to show the case $N=5$, we need the following lemma:
\begin{lemma} \label{gs:lem1}
Let $\l>0$. Assume that $\vec{\f} \in (H^1, \R)^2 \setminus \vec{0}$. Then the following properties hold:
\begin{enumerate}
\renewcommand{\labelenumi}{(\roman{enumi})}
\item There exists a unique $\l_{0} = \l_{0}(\vec{\f}) >0$ such that $\l_{0}^{\frac{5}2} \vec{\f}(\l_{0} \cdot ) \in \m{K}$. 
\item $\l_{0}(\vec{\f}) <1$ if and only if $K(\vec{\f}) <0$. 
\item $\l_{0}(\vec{\f}) =1$ if and only if $\vec{\f} \in \m{K}$. 
\item If $P(\vec{\f}) >0$, then $J_{\om}(\l^{\frac{5}2} \vec{\f}(\l \cdot)) < J_{\om}( \l_{0}^{\frac{5}2} \vec{\f}(\l_{0} \cdot))$ for any $\l>0$, $\l \neq \l_0 = \l_{0}(\vec{\f})$. 
\item $\frac{d}{d \l} J_{\om}(\l^{\frac{5}2} \vec{\f}(\l \cdot)) = \frac{1}{2\l} K(\l^{\frac{5}2} \vec{\f}(\l \cdot))$ for any $\l>0$. 
\item $|\l^{\frac{5}2} v(\l \cdot)|^{\ast} = \l^{\frac{5}{2}}|v(\l \cdot)|^{\ast}$ for any $\l>0$ and all $v \in H^{1}(\R^5)$, where $\ast$ is the Schwarz symmetrization. 
\item If $v_m \rightarrow v$ weakly in $H^1(\R^5)$ and strongly in $L^{3}(\R^5)$, then $\l^{\frac{5}2}v_m(\l \cdot) \rightarrow \l^{\frac{5}2}v(\l \cdot)$ weakly in $H^1(\R^5)$ and strongly in $L^{3}(\R^5)$ for every $\l>0$. 
\end{enumerate}
\end{lemma}
\begin{proof}[Proof of Lemma \ref{gs:lem1}]
Once $\l_{0}(\f)$ is defined by 
\begin{align*}
	(\l_{0}(\f))^{\frac{1}{2}} = \frac{4}{5}L(\vec{\f})P(\vec{\f})^{-1}, 
\end{align*}
the proof is the same as in that of \cite[Lemma 8.2.5]{C2}, so we omit the detail.
\end{proof}

\begin{proof}[Proof of Proposition \ref{gs:t0} in $N=5$]
Let us first show that $\m{M}_{\om}$ is nonempty as in \cite[section 8.2]{C2}.
By Lemma \ref{gs:lem1} (i), $\m{K}$ is nonempty. Hence, \eqref{mini:1} has a minimizing sequence $\{ \vec{\f}_{m} \}_{m=1}^{\I}$.
In particular, $K(\vec{\f}_{m})=0$ and $J_{\om}(\vec{\f}_{m}) \to d_{\om}^1$ as $m \to \I$. Let $w_{jm} = |\f_{jm}|^{\ast}$ and $\g_{jm} = w_{jm}^{\l_{0}(\vec{w}_m)}$ for $j=1$, $2$, where $f^{\l} = \l^{\frac{5}{2}}f(\l x)$ for any $\l>0$ and all function $f(x)$. 
Then, since $\vec{\g}_{m}$ is radially symmetric, Lemma \ref{gs:lem1} (i) implies $\vec{\g}_{m} \in \m{K}$.
Further, we see from Lemma \ref{gs:lem1} (vi) that $\g_{jm} = |\f_{jm}^{\l_{0}(\vec{w}_m)}|^{\ast}$. 
By means of the property of the Schwarz symmetrization (e.g. \cite{LL1}), one has 
\begin{align}
	\begin{aligned}
	&{} P\( |u_1|^{\ast}, |u_2|^{\ast} \) \ge \( |u_1|^2, |u_2| \)_{L^2} \ge P(\vec{u}),  \\
	&{} L(|u_1|^{\ast}, |u_2|^{\ast}) \le L(\vec{u}), \quad M_{\om}\(|u_1|^{\ast}, |u_2|^{\ast} \) = M_{\om}(\vec{u})
	\end{aligned}
	\label{ssy:1}
\end{align}
for any $\vec{u} \in \m{H}^1$, because of $|u_1|$, $|u_2| \ge 0$.
Combining the above with Lemma \ref{gs:lem1} (iii) and (iv), it holds that
\begin{align}
	\begin{aligned}
	d_{\om}^1 \le{}& J_{\om}( \vec{\g}_{m}) \le J_{\om}(\vec{\f}_{m}^{\l_0(\vec{w}_m)}) \le J_{\om}(\vec{\f}_{m}^{\l_0(\vec{\f}_m)}) = J_{\om}(\vec{\f}_{m}),
	\end{aligned}
	\label{gs:3}
\end{align}
where $\vec{\f}_{m}^{\l_0} = (\f_{1m}^{\l_0}, \f_{2m}^{\l_0})$. This yields $J_{\om}(\vec{\g}_{m}) \to d_{\om}^1$ as $m \to \I$. Therefore, $\{ \vec{\g}_{m}\}$ is a nonnegative, radially symmetric, non-increasing minimizing sequence of \eqref{mini:1}. 
%We here remark that $\vec{\g}_{m} \in \m{K}$ provides 
%\begin{align}
%	\begin{aligned}
%	J_{\om}(\vec{\g}_{m}) ={}& \(\frac12 - \frac{2}{N}\)L(\vec{\g}_{m}) + M_{\om}(\vec{\g}_{m}).
%%	&{}+ \( 1 -\om^2 \) \Lebn{\g_{1m}}{2}^2 + \frac{\ka^2 -4\om^2}{2} \Lebn{\g_{2m}}{2}^2.
%	\end{aligned}
%	\label{gs:4}
%\end{align}
Also, since $\{ \vec{\g}_{m} \}$ converges in $\m{H}^1$, $\{ \vec{\g}_{m} \}$ is bounded in $\m{H}^1$. 
Combining this fact with $K(\vec{\g}_{m})=0$ and Sobolev embedding, together with the Young inequality 
\[
	a^2 b \le 2^{\frac{N}6} a^3 + 2^{-\frac{N}{12}}b^3 = 2^{\frac{N}6}\(a^3 + 2^{-\frac{N}4} b^3\)
\]
for any $a$, $b \ge 0$, we have
\begin{align*}
	&{}L(\vec{\g}_{m}) \\
	={}& \frac{N}4 P(\vec{\g}_{m}) \\
	\le{}& C\( \Lebn{\g_{1m}}{3}^3 + \frac{1}{2^{\frac{N}4}} \Lebn{\g_{2m}}{3}^3 \) \\
	\le{}& C\( \(\Lebn{\g_{1m}}{2}^{1-\frac{N}6} \Lebn{\n \g_{1m}}{2}^{\frac{N}6} \)^3 + \frac{1}{2^{\frac{N}4}} \(\Lebn{\g_{1m}}{2}^{1-\frac{N}6} \Lebn{\n \g_{1m}}{2}^{\frac{N}6} \)^3 \) \\
	\le{}& C\( \Lebn{\n \g_{1m}}{2}^{\frac{N}2} + \frac{1}{2^{\frac{N}4}} \Lebn{\n \g_{2m}}{2}^{\frac{N}2} \).
\end{align*}
This implies that
\[
	L(\vec{\g}_{m}) \le C L(\vec{\g}_{m})^{\frac{N}4}.
\]
We see from $N=5$ that there exists $C_0 >0$ such that
\[
	C_{0} \le L(\vec{\g}_{m})
\]
for any $m \in \N$. 
By using $K(\vec{\g}_{m}) =0$ again, together with the above, one has
\begin{align}
	0< \frac{N}{4}C_0 \le P(\vec{\g}_{m}) \le C \( \Lebn{\g_{1m}}{3}^3 + \Lebn{\g_{2m}}{3}^3 \) \label{gs:5}
\end{align}
for all $m \in \N$. 
By means of the Strauss compact embedding $H^1_{\rm{rad}}(\R^N) \hookrightarrow L^{q}_{\rm{rad}}(\R^N)$ for any $q \in (2, 2+ 4/(N-2))$ (see \cite{St}), there exist $\vec{\Psi} \in (H^1_{\rm{rad}}(\R^N))^2$ and subsequences $\{\vec{\g}_{m}\}_{m=1}^{\I}$(we still use the same notation) such that $\vec{\g}_{m} \to \vec{\Psi}$ as $m \to \I$ weakly in $(H^1_{\rm{rad}}(\R^N))^2$ and strongly in $(L^{3}_{\rm{rad}}(\R^N))^2$, respectively. 
Also, \eqref{gs:5} gives us $\vec{\Psi} \neq \vec{0}$.
Set $\vec{\Psi}_0 = \vec{\Psi}^{\l_0(\vec{\Psi})}$.  
By Lemma \ref{gs:lem1} (i), $\vec{\Psi}_0 \in \m{K}$. Moreover, Lemma \ref{gs:lem1} (vii) tells us that $\vec{\g}_{m}^{\l_0(\vec{\Psi})} \to \vec{\Psi}_0$ weakly in $H^1_{\rm{rad}}$ and strongly in $L^3_{\rm{rad}}$ as $m \to \I$.
Hence, collecting Lemma \ref{gs:lem1} (iii), (iv), it follows from \eqref{gs:3} that
\begin{align*}
	d_{\om}^1 \le J_{\om}(\vec{\Psi}_0) \le{}& \liminf_{m \to \I}J_{\om}(\vec{\g}_{m}^{\l_0(\vec{\Psi})} ) \\
	\le{}& \liminf_{m \to \I}J_{\om}(\vec{\g}_{m}^{\l_0(\vec{\g}_{m})}) \\
	={}& \liminf_{m \to \I}J_{\om}(\vec{\g}_{m}) \\
	\le{}& \liminf_{m \to \I}J_{\om}(\vec{\f}_m) = d_{\om}^1,
\end{align*}
which implies $J_{\om}(\vec{\Psi}_0) = d_{\om}^1$. Hence, since $\vec{\Psi}_0 \in \m{M}_{\om}$, $\m{M}_{\om}$ is nonempty.

We shall next prove that $\m{M}_{\om} \subset \m{C}_{\om}$. 
Let $\vec{\g} \in \m{M}_{\om}$.
We denote $\vec{\g}(x) = \s^2 \vec{\g}^{\s}(\s x)$ for any $\s>0$. A computation shows 
\[
	K(\vec{\g}^{\s}) = \s^{-6+N}K(\vec{\g}) = 0,
\]
which implies $\vec{\g}^{\s} \in \m{K}$. Set $f(\s) = J_{\om}(\vec{\g}^{\s})$. We here remark that
\begin{align*}
	f(\s) ={}& \frac{\s^{N-6}}{2} L(\vec{\g}) + \frac{\s^{N-4}}{2} M_{\om}(\vec{\g}) - \frac{\s^{N-6}}{2} P(\vec{\g}) \\
	={}& \s^{N-6}\(\frac12 - \frac{2}{N}\) L(\vec{\g}) + \frac{\s^{N-4}}{2} M_{\om}(\vec{\g}).
\end{align*}
Since $\vec{\g}^1 = \vec{\g} \in \m{M}_{\om}$, we have $f'(1)=0$. 
By using $K(\vec{\g}) = 0$, one sees that
\begin{align*}
	f'(1) ={}& \frac{N-4}{2} \( \Lebn{\n \g_1}{2}^2 + (1 - \om^2)\Lebn{\g_1}{2}^2 - \( \g_1^2,  \g_2 \)_{L^2} \) \\
	&{}+ \frac{N-4}{4} \( \Lebn{\n \g_2}{2}^2 + (\ka^2 - 4\om^2)\Lebn{\g_2}{2}^2 - \(  \g_1^2, \g_2 \)_{L^2} \) \\
	={}& \frac{N-4}{2} \( \langle \pa_{u_1}J_{\om}(\vec{\g}), \g_1 \rangle_{H^{-1} \times H^1} + \frac{1}{2} \langle \pa_{u_2}J_{\om}(\vec{\g}), \g_2 \rangle_{H^{-1} \times H^1} \).
\end{align*}
Combining these above, we see that
\begin{align}
	\langle \pa_{u_1}J_{\om}(\vec{\g}), \g_1 \rangle_{H^{-1} \times H^1}  + \frac12 \langle \pa_{u_2}J_{\om}(\vec{\g}), \g_2 \rangle_{H^{-1} \times H^1} =0. \label{gs:7}
\end{align}
On the other hand, one has
\begin{align*}
	\langle \pa_{u_1}K(\g_1, \g_2), \g_1 \rangle_{H^{-1} \times H^1} ={}& \langle -4\D \g_1 -N \g_1 \g_2, \g_1 \rangle_{H^{-1} \times H^1} \\
	={}& 4 \Lebn{\n \g_1}{2}^2 - N P(\g_1, \g_2), \\
	\langle \pa_{u_2}K(\g_1, \g_2), \g_2 \rangle_{H^{-1} \times H^1} ={}& \frac12 \langle -4\D \g_2 -N \g_1^2, \g_2 \rangle_{H^{-1} \times H^1} \\
	={}& 2 \Lebn{\n \g_2}{2}^2 - \frac{N}{2} P(\g_1, \g_2).
\end{align*}
It follows from $K(\vec{\g}) =0$ that
\begin{align}
	\begin{aligned}
	\langle \pa_{u_1}K(\vec{\g}), u_1 \rangle_{H^{-1} \times H^1} ={}& -2 \Lebn{\n \g_2}{2}^2, \\
	\langle \pa_{u_2}K(\vec{\g}), u_1 \rangle_{H^{-1} \times H^1} ={}& -2 \Lebn{\n \g_1}{2}^2 + \Lebn{\n \g_2}{2}^2.
	\end{aligned}
	\label{gs:6}
\end{align}
Since $\vec{\g} \in \m{M}_{\om}$, there exists a Lagrange multiplier $\l \in \R$ such that
\[
	\begin{pmatrix}
	\pa_{u_1}J_{\om}(\vec{\g}) \\
	\pa_{u_2}J_{\om}(\vec{\g}) 
	\end{pmatrix}
	= \l
	\begin{pmatrix}
	\pa_{u_1}K(\vec{\g}) \\
	\pa_{u_2}K(\vec{\g})
	\end{pmatrix}
	.
\]
Unifying \eqref{gs:7} and \eqref{gs:6}, together with the above, we have 
\[
	\l \(-\frac32 \Lebn{\n \g_2}{2}^2 - \Lebn{\n \g_1}{2}^2 \) = 0,
\]
which yields $\l=0$. Thus, since $\pa_{u_1}J_{\om}(\vec{\g}) =0$ and $\pa_{u_2}J_{\om}(\vec{\g}) =0$, $\m{M}_{\om} \subset \m{C}_{\om}$.

Finally, we shall show $\m{G}_{\om} = \m{M}_{\om}$. Set
\[
	l = \min \{J_{\om}(\vec{\f})\ |\ \vec{\f} \in C_{\om} \}
\]
and take $\vec{\g} \in \m{G}_{\om}$. Then $J_{\om}(\vec{\g}) =l$.  
Thanks to Lemma \ref{lgs:0} (iv) and $\vec{\g} \in \m{G}_{\om}$, $\vec{\g} \in \m{K}$ is valid.

From $J_{\om}(\vec{\g}) = l$ and $\vec{\g} \in \m{K}$, we have $l \ge d_{\om}^1$. In order to show $d_{\om}^1 \ge l$, let us take $\vec{\f}_{\om} \in \m{M}_{\om}$. Since $\m{M}_{\om} \subset \m{C}_{\om}$, taking the definition of $l$ into account, we see from $J_{\om}(\vec{\f}_{\om}) = d_{\om}^1$ that $d_{\om}^1 \ge l$. Hence $d_{\om}^1 =l$ is concluded. The equivalence of the two problems is immediate. This completes the proof.
\end{proof}

Let us prove the case $N=4$. We first remark the following:
\begin{theorem}[\cite{HOT}] \label{hot:1}
Let $N=4$. Fix $\om \in \R$ with $|\om| < \min\(1, \ka/2 \)$. There exists a pair of non-negative, radially symmetric function $\vec{\f}_0 \in \m{P}$ such that
\[
	I_{\om}(\vec{\f}_0) = \a_1 = \inf \{I_{\om}(\vec{\f})\ |\ \vec{\f} \in \m{P} \},
\]
where $\m{P} = \{\vec{\f} \in \( H^1(\R^{4}, \R) \)^2 \setminus \vec{0}\ |\ P(\vec{\f})>0 \}$.
\end{theorem}

Following \cite[Proposition 2.5]{Na}, we shall show Proposition \ref{gs:t0} in $N=4$.

\begin{proof}[Proof of Proposition \ref{gs:t0} in $N=4$.]
Let us first handle the constrained minimization problem 
\begin{align}
d_{\om}^2 := \inf \left\{ J_{\om}^2 (\vec{\f})\ \left|\ \vec{\f} \in \m{H}^1 \setminus \vec{0}, K(\vec{\f}) \le 0 \right. \right\}, \label{cri:4}
\end{align}
%\begin{align}
%    d_{\om}^2 = \inf \{J_{\om}^2 (\vec{\f});\; (\vec{\f}) \in (H^1(\R^4))^2 \setminus \{(0, 0)^2\}, K(\vec{\f}) \le 0 \}, \label{cri:4}  
%\end{align}
where 
\begin{align*}
	J_{\omega}^2(\vec{\f}) = &{} M_{\om}(\vec{\f}).
%	\label{cri:5}
\end{align*}
Let $\{\vec{\f}_{n}\} \subset \m{H}^1$ be a minimizing sequence for \eqref{cri:4}. Namely, $\{\vec{\f}_{n}\}$ satisfies 
\begin{align}
    &{} \lim_{n \to \I} J_{\om}^2 (\vec{\f}_{n}) = d_{\om}^2, \nonumber \\ %\label{cri:6}
    &{} K(\vec{\f}_{n}) \le 0, \quad n \in \Z_+. \label{cri:7}
\end{align}

We here employ the following result which is so-called linear profile decomposition:
\begin{proposition}\label{p:d41}
Let $\{ \vec{u}_n\}$ be a bounded sequence in $\m{H}^1(\R^4)$ such that 
\begin{align*}
    \limsup_{n \to \I} \Lebn{\vec{u}_n}{3}^3 \ge C %\label{d4:0}
\end{align*}
for some $C>0$. Then there exist $\vec{v} \in \m{H}^1(\R^4)$ and a sequence $\{ y_n \} \subset \R^{N}$ satisfying the following: There exists a sub-sequence of $\{ \vec{u}_n\}$ (we denote it by the same notation) such that
\begin{align}
    &{} \vec{v}_{n} := \vec{u}_{n}(\cdot + y_n) \to \vec{v}\ \not\equiv \vec{0} \quad \text{weakly in $\m{H}^1(\R^4)$}, \label{d4:1} \\
    &{} \lim_{n \to \I} \{ K(\vec{v}_{n}) - K(\vec{v}_{n}-\vec{v}) - K(\vec{v}) \} = 0, \label{d4:2} \\
    &{} \lim_{n \to \I} \{ J_{\om}^2(\vec{v}_{n}) - J_{\om}^2(\vec{v}_{n}-\vec{v}) - J_{\om}^2(\vec{v}) \} = 0 \label{d4:3}.
\end{align}
\end{proposition}
We postpone the proof of Proposition \ref{p:d41} and continue to prove Proposition \ref{gs:t0}.
In order to employ Proposition \ref{p:d41}, we shall construct a bounded sequence in $\m{H}^1(\R^4)$. Let us define the scaled function
\begin{align*}
    \vec{\g}_{n}(x) = \mu_n^2 \vec{\f}_{n}(\mu_{n} x), \quad \mu_n = \(2 P(\vec{\f}_{n}) \)^{-\frac12}.
\end{align*}
Note that \eqref{cri:7} gives us $P(\vec{\f}_n) >0$. Then we have
\begin{align*}
    &{} J_{\om}^2(\vec{\g}_{n})  = J_{\om}^2(\vec{\f}_{n}) \to d_{\om}^2 \quad (n \to \I), \\
    &{} 2 P( \vec{\g}_{n} ) = 2 \mu^{2}_n P( \vec{\f}_{n} ) = 1, \\
    &{} K(\vec{\g}_{n}) = \mu_n^2 K(\vec{\f}_{n}) \le 0.
%    \label{cri:8}
\end{align*}
This implies that $\{ \vec{\g}_{n} \}$ is $L^2 \times L^2$-bounded and $L(\vec{\g}_{n}) \le 1/2$.
Thus, $\{ \vec{\g}_{n} \}$ is the $\m{H}^1$-bounded minimizing sequence for \eqref{cri:4}.
By using the Young inequality, we see that for any $n \in \N$,
\[
	\frac12 \le \Lebn{\g_{1n}}{3}^3 + \Lebn{\g_{2n}}{3}^3.
\]
Therefore, applying Proposition \ref{p:d41} to $\{ \vec{\g}_{n} \}$, there exist $\vec{\g}^1 \in \m{H}^1$ and a subsequence of $\{ \vec{\g}_{n} \}$ still denoted by the same notation such that 
\begin{align}
    &{} \vec{\g}_{n}^1 := \vec{\g}_{n}(\cdot + y_n^1) \to \vec{\g}^1\ \not\equiv \vec{0} \quad \text{weakly in $\m{H}^1(\R^4)$},\label{cri:9} \\
    &{} \lim_{n \to \I} \{ K(\vec{\g}_{n}^1) - K(\vec{\g}_{n}^1 - \vec{\g}^1) - K(\vec{\g}^1) \} = 0, \label{cri:10} \\
    &{} \lim_{n \to \I} \{ J_{\om}^2(\vec{\g}_{n}^1) - J_{\om}^2(\vec{\g}_{n}^1 - \vec{\g}^1) - J_{\om}^2(\vec{\g}^1) \} = 0 \label{cri:11}
\end{align}
for some $\{y_n^1\} \subset \R^4$.
Let us assume $K(\vec{\g}^1)>0$ to show $K(\vec{\g}^1) \le 0$. By using \eqref{cri:10} and $K(\vec{\g}_{n}^1) \le 0$, we deduce that
\[
    K(\vec{\g}_{n}^1 - \vec{\g}^1) \le K(\vec{\g}_{n}^1) - K(\vec{\g}^1) \le 0
\]
for sufficiently large $n$. Therefore, it follows from the definition of $d_{\om}^2$ that
\[
    J_{\om}^2 (\vec{\g}_{n}^1 - \vec{\g}^1) \ge d_{\om}^2 
\]
for sufficiently large $n$. Hence, combining the above with \eqref{cri:11}, one sees that
\begin{align*}
    J_{\om}^2 (\vec{\g}^1) \le{}& J_{\om}^2(\vec{\g}_{n}^1) - J_{\om}^2(\vec{\g}_{n}^1 - \vec{\g}^1) \le J_{\om}^2(\vec{\g}_{n}^1) - d_{\om}^2
\end{align*}
for sufficiently large $n$. From $\lim_{n \to \I} J_{\om}^2(\vec{\g}_{n}^1) = d_{\om}^2$, taking $n \to \I$, this tells us that $J_{\om}^2(\vec{\g}^1) \le 0$, which contradicts $\vec{\g}^1 \not\equiv \vec{0}$. Thus we have
\begin{align}
    K(\vec{\g}^1) \le 0. \label{cri:12}
\end{align}
By using the definition of $d_{\om}^2$, together with \eqref{cri:12}, one obtains $d_{\om}^2 \le J_{\om}^2(\vec{\g}^1)$. Further, \eqref{cri:9} gives us 
\[
    \Lebn{\g_{j}^1}{2} \le \liminf_{n \to \I}\Lebn{\g_{jn}^1}{2}
\]
for each $j=1$, $2$, which yields
\[
    J_{\om}^2(\vec{\g}^1) \le \liminf_{n \to \I} J_{\om}^2(\vec{\g}_{n}^1) = d_{\om}^2.
\]
Hence we conclude
\begin{align}
    d_{\om}^2 = J_{\om}^2(\vec{\g}^1). \label{cri:13}
\end{align}
Thus $\vec{\g}^1$ is a solution of \eqref{cri:4}. 
Set $\vec{\Psi}^1 = \( |\g_1^1|^{\ast}, |\g_2^1|^{\ast} \)$.
Together with \eqref{cri:13}, it follows from \eqref{cri:12} and \eqref{ssy:1} that
\begin{align*}
	K(\vec{\Psi}^1) \le K(\vec{\g}^1) \le 0, \quad J_{\om}^2(\vec{\Psi}^1) = J_{\om}^2(\vec{\g}^1) \le d_{\om}^2.
\end{align*}
Further, $K(\vec{\Psi}^1) \le 0$ tells us $d_{\om}^2 \le J_{\om}^2(\vec{\Psi}^1)$. Thus, since $d_{\om}^2 = J_{\om}^2(\vec{\Psi}^1)$, $\vec{\Psi}^1$ is a radial solution to \eqref{cri:4}. 

Let us next show $\vec{\Psi}^1$ is a solution of \eqref{mini:1}.
Set
\begin{align}
    s = s(\vec{\Psi}^1) = \(\frac{L(\vec{\Psi}^1)}{P(\vec{\Psi}^1)}\)^{\frac12}. \label{nawa:1}
\end{align}
By \eqref{cri:12}, $s \le 1$.
Let $\vec{\Psi}_{s} = \vec{\Psi}^1\(\frac{\cdot}{s}\)$. Then we have
\begin{align*}
    &{} K(\vec{\Psi}_{s}) = 2s^2\( L(\vec{\Psi}^1) -s^2 P(\vec{\Psi}^1) \) = 0, \\
    &{}J_{\om}^2(\vec{\Psi}_{s}) = s^4 J_{\om}^2 (\vec{\Psi}^1) \le J_{\om}^2(\vec{\Psi}^1) = d_{\om}^2.
%    \label{cri:14}
\end{align*}
It comes from $K(\vec{\Psi}_{s}) = 0$ that $d_{\om}^2 \le J_{\om}^2(\vec{\Psi}_{s})$. Hence we obtain
$d_{\om}^2 = J_{\om}^2(\vec{\Psi}_{s})$. 
Thus $s^4 d_{\om}^2 = d_{\om}^2$. This yields $s=1$. Therefore it is concluded that 
\begin{align}
    K(\vec{\Psi}^1) = 0. \label{cri:15}
\end{align}
We shall prove $d_{\om}^1 = d_{\om}^2$. Remark that $J_{\om}(\vec{u}) = J_{\om}^2(\vec{u})$ as long as $K(\vec{u}) =0$.
By the definition of $d_{\om}^1$ and $d_{\om}^2$, it holds that $d_{\om}^2 \le d_{\om}^1$. Also, combining $d_{\om}^2 = J_{\om}^2(\vec{\Psi}^1)$ with \eqref{cri:15}, we have $d_{\om}^2 \ge d_{\om}^1$. Thus $d_{\om}^1 = d_{\om}^2$. 
It is then established that $\vec{\Psi}^1$ is a solution of \eqref{mini:1}. Hence $\vec{\Psi}^1 \in \m{M}_{\om}$, that is, $\m{M}_{\om}$ is nonempty.

In order to show $\m{M}_{\om} \subset \m{C}_{\om}$,
let us next handle the minimization problem
\begin{align*}
	\a_1 = \inf \{I_{\om}(\vec{\f})\ |\ \vec{\f} \in \m{P} \}. %\label{mini:na}
\end{align*}
Set
\begin{align*}
	\m{A}_{\om} = \{\vec{\f} \in \m{P}\ |\ \a_1 = I_{\om}(\vec{\f}) \}. %\label{mini:na1}
\end{align*}
We first shall prove $\m{M}_{\om} \subset \m{A}_{\om}$. $\m{K} \subset \m{P}$ is trivial.
In view of Theorem \ref{hot:1}, we can construct a minimizing sequence $\{\vec{\f}_{n}\} \subset \m{P}$ for $I_{\om}$. Namely, $\{\vec{\f}_{n}\}$ satisfies 
\[
	\lim_{n \to \I} I_{\om}(\vec{\f}_{n}) = \a_1, \quad P(\vec{\f}_{n}) >0
\]
for any $n \in \Z_+$. Set $\vec{\f}_n^0 = \( |\f_{1n}|^{\ast}, |\f_{2n}|^{\ast} \)$. Thanks to the properties of the Schwarz symmetrization, it holds that $I_{\om}(\vec{\f}_n^0) \le I_{\om}(\vec{\f}_{n})$ and $P(\vec{\f}_{n}) >0$. We then have $\lim_{n \to \I} I_{\om}(\vec{\f}_{n}^0) = \a_1$.

We set $s = s(\vec{\f}_n^0)$ as in \eqref{nawa:1}. Letting $\vec{\Phi}_{n} = \vec{\f}_n^0 \(\frac{\cdot}{s}\)$, it holds that
\begin{align}
	K(\vec{\Phi}_{n}) ={}& 0 \Leftrightarrow L(\vec{\Phi}_{n}) = P(\vec{\Phi}_{n}), \label{nawa:2} \\
	I_{\om}(\vec{\Phi}_{n}) ={}& I_{\om}(\vec{\f}_{n}^0). \label{nawa:3}
\end{align}
By \eqref{nawa:3}, we have
\begin{align*}
	\lim_{n \to \I} I_{\om}(\vec{\Phi}_{n}) = \a_1.
\end{align*}
Thanks to \eqref{nawa:2}, because of the definition $d_{\om}^1$, it is deduced that
\[
	I_{\om}(\vec{\Phi}_{n}) = M_{\om}(\vec{\Phi}_{n})^{1/2} \ge \sqrt{2} \(d_{\om}^1\)^{1/2}.
\]
Taking $n \to \I$, this yields $\a_1 \ge \sqrt{2} \(d_{\om}^1\)^{1/2}$. Whereas, for any $\vec{\g} \in \m{M}_{\om}$, noting $\m{K} \subset \m{P}$, we have
\[
	\a_1 = \inf \{I_{\om}(\vec{\f})\ |\ \vec{\f} \in \m{P} \} \le I_{\om}(\vec{\g}) = M_{\om}(\vec{\g})^{1/2} = \sqrt{2}  \(d_{\om}^1\)^{1/2}.
\]
Hence it holds that $\a_1 = \sqrt{2} \(d_{\om}^1\)^{1/2} =  I_{\om}(\vec{\g})$ for any $\vec{\g} \in \m{M}_{\om}$.
%from \eqref{cri:15} 
%\[
%	 = I_{\om}(\vec{\g}^1) \( = I_{\om}(\g_1^1, \g_2^1) \).
%\]
%Thus $\vec{\g}^1 =  (\g_1^1, \g_2^1) \in \m{A}_{\om}$. This implies that $\m{A}_{\om}$ is nonempty.
Thus $\m{M}_{\om} \subset \m{A}_{\om}$.
Let us prove $\m{M}_{\om} \subset \m{C}_{\om}$.
Take $\vec{\g} \in \m{M}_{\om}$.
Set
\begin{align}
	\vec{\Psi}(x) = \frac1{2\a_1} \vec{\g}\(\frac{x}{\sqrt{2}} \) \label{scale:1}
\end{align}
It then follows from $\m{M}_{\om} \subset \m{A}_{\om}$ that
\[
	I_{\om}( \vec{\Psi} ) = I_{\om}(\vec{\g}) = \a_1, \quad \vec{\Psi} \in \m{P},
\]
which yields $\vec{\Psi} \in \m{A}_{\om}$. 
Since $\vec{\Psi}$ is a critical point for $I_{\om}$, that is
\[
	\left. \frac{d}{ds}I_{\om}\( \vec{\Psi} + s \vec{u} \) \right|_{s=0} = 0 
\]
for any $\vec{u} \in \m{H}^1$, we obtain
\begin{align*}
	\begin{cases}
	\ds  - 2\D \Psi_1 + (1 -\om^2) \Psi_1 = 2\a_1 \Psi_1 \Psi_2, \\
	\ds  -\D \Psi_2 + \frac{\ka^2- 4\om^2}{2} \Psi_2 = \a_1 (\Psi_1)^2,
	\end{cases}
\end{align*}
in $H^{-1}$. By the scaling \eqref{scale:1}, we deduce that $\vec{\g}$ is a solution to \eqref{snlkg}, that is $\m{M}_{\om} \subset \m{C}_{\om}$.
A remaining proof is the same to the case $N=5$. This completes the proof.

\end{proof}

Let us finally prove Proposition \ref{p:d41}. We need the following lemma to show Proposition \ref{p:d41}.

\begin{lemma}[Lieb's compactness Theorem, \cite{L}]
\label{lem2-1}
Let $\{ \vec{u}_{n} \}$ be a bounded sequence in $\m{H}^1(\R^N)$ with $\inf_{n\in \Z_+}\|\vec{u}_n\|_{L^q(\mathbb{R}^N)}>0$ for some $q\in (2, 2+ 4/(N-2))$. Then there exist $\{y_n\} \subset \R^N$, $\vec{w} \in \m{H}^1 \backslash \vec{0}$, and a subsequence $\{n_j\} \subset \Z_+$ such that
\[
         \vec{u}_{n_j}(\cdot + y_{n_j}) \rightarrow \vec{w},\; \text{weakly}\ \text{in}\ \m{H}^1(\R^N)
\]
as $j \to \I$.
\end{lemma}

\begin{proof}[Proof of Proposition \ref{p:d41}]
Firstly, we have \eqref{d4:1} immediately from Lemma \ref{lem2-1}. Also, by \eqref{d4:1}, it is easy to show \eqref{d4:3} because 
\begin{align*}
    &{}\Lebn{v_{jn}}{2}^2 - \Lebn{v_{jn} - v_j}{2}^2 - \Lebn{v_j}{2}^2 = 2\( v_{jn}, v_j \)_{L^2} -2\Lebn{v_j}{2}^2 \to 0
\end{align*}
as $n \to \I$ for $j=1$, $2$. Let us prove \eqref{d4:2}. Arguing as in the proof of \eqref{d4:3}, it holds that 
\[
    \Lebn{\n v_{jn}}{2}^2 - \Lebn{\n v_{jn} - \n v_j}{2}^2 - \Lebn{\n v_j}{2}^2 \to 0
\]
as $n \to \I$ for $j=1$, $2$. Next, we shall show 
\begin{align}
    \re \(v_{1n}^2, v_{2n}\)_{L^2} - \re \( (v_{1n}- v_{1})^2, v_{2n}-v_{2} \)_{L^2} - \re \(v_{1}^2, v_{2}\)_{L^2} \to 0 \label{d4:4}
\end{align}
as $n \to \I$. A direct computation shows that
\begin{align*}
    &{}\left| \re \(v_{1n}^2, v_{2n}\)_{L^2} - \re \( (v_{1n}- v_{1})^2, v_{2n}-v_{2} \)_{L^2} - \re \(v_{1}^2, v_{2}\)_{L^2} \right| \\
    \le{}& \left| \( (v_{1n}-v_1)(v_{1n}+v_1), v_2 \)_{L^2}\right| + \left| \(v_{1}^2, v_2-v_{2n} \)_{L^2}\right| + 2\left| \(v_{1n}v_1, v_{2n} - v_2 \)_{L^2}\right|.
\end{align*}
By the Rellich-Kondrachov theorem $H^1(Q_y) \hookrightarrow L^q(Q_y)$ for any $y \in \Z^N$, where $q \in [2, 2+ 4/(N-2))$ and $Q_y$ is defined by
\begin{align*}
        Q_y = \{x=(x_1,\ldots,x_N)\in\mathbb{R}^N:y_j<x_j<y_j+1\ (j=1,\ldots,N)\} %\label{cpt:1}
\end{align*}
(e.g. \cite[p168]{AF1}), we have
$\vec{v}_n \to \vec{v}$ a.e in $\R^N$ as $n \to \I$. Hence, by using the boundedness of $\{\vec{v}_n\}$ in $\m{H}^1(\R^N)$ and  Sobolev embedding, together with the Lebesgue convergence theorem, \eqref{d4:4} is obtained. Thus \eqref{d4:2} holds. This completes the proof. 
\end{proof}

We finish this section by proving Proposition \ref{gs:t0} in $N = 2$, $3$.

\begin{proof}[Proof of Proposition \ref{gs:t0} in $N=2$, $3$]
Let us first show that $\m{M}_{\om}$ is nonempty.
We here remark that
\begin{align}
	&{}J_{\om}(\vec{\f}) - \frac{K_{\om}^0(\vec{\f})}{2(\a+2)} = \frac{M_{\om}(\vec{\f})}{\a+2}  
	\label{lem:52} \\
	d_{\om}^0 ={}&  \inf \left\{ \left.\frac{M_{\om}(\vec{\f})}{\a+2}  \right.|\ \vec{\f} \in (H^1(\R^{N}))^2 \setminus \vec{0},\ K_{\om}^0(\vec{\f}) = 0 \right\}. 
	\label{lem:53}
\end{align}
\eqref{lem:53} gives us $d_{\om}^0 \ge 0$. 
Let $\{ \vec{\f}_m \} \subset \m{H}^1(\R^N)$ be a minimizing sequence for \eqref{mini:2}. By considering the  Schwarz symmetrization of $\vec{\f}_{m}$, we may assume $\{ \vec{\f}_m \} \subset (H^{1}_{rad}(\R^N))^2$. From \eqref{lem:53}, there exists $C>0$ such that 
\[
	M_{\om}(\vec{\f}_m) \le C
\]
for any $m \in \Z^{+}$. This implies that $\{\vec{\f}_{m}\}$ is bounded in $\m{H}^0(\R^N)$. Hence, combining $K^0_{\om}(\vec{\f}_m) = 0$ with the Gagliardo-Nirenberg inequality, together with the Young inequality, we see that
\begin{align*}
%	&{}\a \( (m^2_1- \omega^2) \Lebn{u_1}{2}^2 + \frac{m^2_2 - 4\omega^2}{2} \Lebn{u_2}{2}^2 \) \\
	L(\vec{\f}_m)  \le{}& (\a+2) P( \vec{\f}_{m} ) \\
	\le{}& C\( \Lebn{\f_{1m}}{3}^3 + \frac{1}{2^{\frac{N}{4}}} \Lebn{\f_{2m}}{3}^3 \) \\
	\le{}& C\( \Lebn{\f_{1m}}{2}^{3-\frac{N}{2}} \Lebn{\n \f_{1m}}{2}^{\frac{N}{2}} + \frac{1}{2^{\frac{N}{4}}} \Lebn{\f_{2m}}{2}^{3-\frac{N}{2}} \Lebn{\n \f_{2m}}{2}^{\frac{N}{2}} \) \\
	\le{}& C L( \vec{\f}_{m} )^{\frac{N}{2}}.
\end{align*}
Since $N=2$, $3$, this allows us to exists $C>0$ such that
\[
	L( \vec{\f}_{m} ) \le C
\]
for any $m \in \Z^+$. Therefore, $\{\vec{\f}_{m}\}$ is bounded in $\m{H}^1(\R^N)$.
Hence, by means of the Strauss compact embedding $H^1_{\rm{rad}}(\R^N) \hookrightarrow L^{q}_{\rm{rad}}(\R^N)$ for any $q \in (2, 2+ 4/(N-2))$ (see \cite{St}), 
there exist a subsequence $\{ \vec{\f}_{m}\}$ (we still use the same notation) and $\vec{w} \in \m{H}^1(\R^N)$ such that $\vec{\f}_{m} \to \vec{w}$ weakly in $\m{H}^1(\R^N)$ and strongly in $(L^3(\R^N))^2$. 

Let us show $\vec{w} \neq \vec{0}$. Suppose that $\vec{w} = \vec{0}$. By using $K_{\om}^0(\vec{\f}_{m})=0$ and $\vec{\f}_{m} \to \vec{0}$ strongly in $(L^3(\R^N))^2$, we deduce that $\vec{\f}_{m} \to 0$ strongly in $\m{H}^1(\R^N)$.
On the other hand, together with the Young inequality, one sees from $K_{\om}^0(\vec{\f}_m)=0$ and Sobolev embedding that 
\begin{align*}
	&{}\a M_{\om}( \vec{\f}_{m}) + (\a+2) L(\vec{\f}_m) \\
	\le{}& C \( \Lebn{\f_{1m}}{3}^3 + \frac{1}{2^{\frac32}}\Lebn{\f_{2m}}{3}^3 \) \\
	\le{}& C\( \Lebn{\f_{1m}}{2}^{2-\frac{N}{3}} \Lebn{\n \f_{1m}}{2}^{\frac{N}{3}} + \frac{1}{2} \Lebn{\f_{2m}}{2}^{2-\frac{N}{3}} \Lebn{\n \f_{2m}}{2}^{\frac{N}{3}} \)^{\frac{3}{2}} \\
	\le{}& C \( \a M_{\om}( \vec{\f}_m ) + (\a+2) L( \vec{\f}_m ) \)^{\frac32}.
\end{align*}
Here, in the above last line, we employ the inequality due to the Young inequality
\[
	\Lebn{\f_{jm}}{2}^{2-\frac{N}{3}} \Lebn{\n \f_{jm}}{2}^{\frac{N}{3}} \le \e \Lebn{\f_{jm}}{2}^2 + \frac{1}{\e^{\frac6N -1}} \Lebn{\n \f_{jm}}{2}^2
\]
for any $\e >0$ and $j=1$, $2$. Unifying the above and $\vec{\f}_{m} \neq \vec{0}$, there exists $C>0$ such that
\begin{align*}
	&{}\a M_{\om}( \vec{\f}_m ) + (\a+2) L(\vec{\f}_m) \ge C
\end{align*}
for any $m \in \Z^+$.
Hence, we have $\Sobn{\f_{1m}}{1}^2 + \frac12 \Sobn{\f_{2m}}{1}^2 \ge C$ for any $m \in \Z^+$. However this contradicts $\vec{w} = \vec{0}$, that is, $\vec{w} \neq \vec{0}$. In particular, $\vec{w} \in (H^1(\R^N))^2 \setminus \vec{0}$. 
Collecting \eqref{lem:52}, \eqref{lem:53} and $K^{0}_{\om}(\vec{\f}_{m}) =0$, we reach to
\begin{align*}
	d_{\om}^0 \le \frac{M_{\om} (\vec{w})}{\a+2} \le \liminf_{m \to \I} \frac{M_{\om} (\vec{\f}_m)}{\a+2} \le \liminf_{m \to \I} J_{\om}(\vec{\f}_m) = d_{\om}^0,
\end{align*}
which yields $\frac{M_{\om} (\vec{w})}{\a+2} = d_{\om}^0$. Also, since
\begin{align*}
	&{}\left| P(\vec{w}) - P(\vec{\f}_m) \right| \\
	\le{}& \Lebn{w_1}{3}^2 \Lebn{w_2 -\f_{2m}}{3} + \(\Lebn{w_1}{3} + \Lebn{\f_{1m}}{3} \)\Lebn{\f_{2m}}{3} \Lebn{w_1 - \f_{1m}}{3} \\
	\to{}& 0
\end{align*}
as $m \to \I$, it is deduced that
\[
	K_{\om}^0 (\vec{w}) \le \liminf_{m \to \I} K_{\om}^0 ( \vec{\f}_m) = 0.
\]
Here, if $K_{\om}^0 (\vec{w}) <0$, then Lemma \ref{lem:4} (i) (will be shown in section \ref{vkl}) leads to
\[
	d_{\om}^0 < \frac{M_{\om}(\vec{w})}{\a+2}.
\]
This yields $K_{\om}^0(\vec{w}) =0$. Hence we conclude that $\vec{w}$ attains \eqref{mini:2}, that is $\m{M}_{\om}$ is nonempty. 

We shall prove $\m{M}_{\om} \subset \m{C}_{\om}$.
Let $\vec{w} \in \m{M}_{\om}$. Then there exists a Lagrange multiplier $\eta \in \R$ such that 
\begin{align}
	\begin{pmatrix}
	\pa_{u_1}J_{\om}(\vec{w}) \\
	\pa_{u_2}J_{\om}(\vec{w}) 
	\end{pmatrix}
	= \frac{\eta}{2(\a+2)}
	\begin{pmatrix}
	\pa_{u_1}K_{\om}^0(\vec{w}) \\
	\pa_{u_2}K_{\om}^0(\vec{w})
	\end{pmatrix}
	. \label{gssub:1}
\end{align}
%where $\pa_{u_j}K^0_{\om}(\vec{w})$ ($j=1$, $2$) are defined by
%\begin{align*}
%	&{}\langle \pa_{u_1}K^0_{\om}(\vec{w}), \f \rangle \\
%	={}& 2\re \int_{\R^N}\(\a(m_1^2 - \om^2)w_1 -(\a+2)\D w_1 - (\a+2)\overline{w_1}w_2 \) \overline{\f}\ dx, \\
%	&{}\langle \pa_{u_2}K^0_{\om}(\vec{w}), \f \rangle \\
%	={}& \re \int_{\R^N}\(\a(m_2^2 - 4\om^2)w_2 -(\a+2)\D w_2 - (\a+2) w_1^2 \) \overline{\f}\ dx
%\end{align*}
%for any $\f \in C_0^{\I}(\R^N)$. 
We see from \eqref{gssub:1} that $\vec{w}$ satisfies 
\begin{align}
	\begin{cases}
	-(1-\eta)\D w_1 + (m_1^2 -\om^2)\(1- \frac{\a}{\a+2} \eta \)w_1 -(1-\eta) \overline{w_1} w_2 =0, \\
	-(1-\eta)\D w_2 + (m_2^2 -4\om^2)\(1- \frac{\a}{\a+2} \eta \)w_2 -(1-\eta) w_1^2 =0
	\end{cases}
	\label{gssub:2}
\end{align}
in $H^{-1}(\R^N)$. 

Let us show $\eta<1$. Suppose that $\eta \ge 1$. \eqref{gssub:2} gives us
\begin{align*}
	0 ={}& (1-\eta) L(\vec{w}) = \(1 -\frac{\a}{\a+2} \eta \) M_{\om}(\vec{w}) - \frac{3}{2}(1-\eta) P(\vec{w}).
\end{align*}
Combining the above with $K_{\om}^0( \vec{w}) =0$ and $\eta \ge 1$, we have
\begin{align*}
	0 = \frac12 (\eta -1) L(\vec{w}) + \frac{N+\eta}{2(\a+2)} M_{\om}(\vec{w}) 
	\ge \frac{M_{\om}(\vec{w})}{2(\a+2)} = \frac{d_{\om}^0}{2} >0.
\end{align*}
This yields a contradiction. Hence $\eta<1$ holds. Therefore, since
\[
	1-\eta>0, \quad 1-\frac{\a}{\a+2}\eta >0,
\]
it follows from the elliptic regularization method that $x \cdot \n w_j \in H^1$ for $j=1$, $2$ (see \cite[Theorem 8.1.1]{C2}). By means of \eqref{gssub:1}, we deduce that 
\begin{align*}
	0 ={}& K_{\om}^0 (\vec{w}) \\
	={}& 2 \pa_{\l} J_{\om} (\l^{2} \vec{w}(\l \cdot))|_{\l=1} \\
	={}& 2 \langle \pa_{u_1}J_{\om}(\vec{w}), x \cdot \n w_1 + 2 w_1 \rangle + 2 \langle \pa_{u_2}J_{\om}(\vec{w}), x \cdot \n w_2 + 2 w_2 \rangle \\
	={}& \frac{\eta}{\a+2} \langle \pa_{u_1}K_{\om}^0(\vec{w}), x \cdot \n w_1 + 2 w_1 \rangle + \frac{\eta}{\a+2} \langle \pa_{u_2}K_{\om}^0(\vec{w}), x \cdot \n w_2 + 2 w_2 \rangle \\
	={}& \frac{\eta}{\a+2} \pa_{\l} K_{\om}^0(\l^{2} \vec{w}(\l \cdot))|_{\l=1}.
\end{align*}
Further, one sees from $K_{\om}^0(\vec{w})=0$ that
\begin{align*}
	\pa_{\l} K_{\om}^0(\l^{2} \vec{w}(\l \cdot))|_{\l=1} ={}& \a^2 M_{\om}(\vec{w}) + (\a+2)^2 \( L(\vec{w}) - P(\vec{w}) \) \\
	={}& -2\a M_{\om}(\vec{w}) <0,
\end{align*}
which implies $\eta =0$. 
Thus, we conclude $\pa_{u_1}J_{\om}(\vec{w})=0$ and $\pa_{u_2}J_{\om}(\vec{w})=0$. Thus $\m{M}_{\om} \subset \m{C}_{\om}$. 
A remaining proof is similar to the case $N=5$. This completes the proof.
\end{proof}

\section{Some of the variational lemmas} \label{vkl}

In this section, we will prove some of the lemmas due to the variational results in section \ref{sec:5}.
We here remark the identity
\begin{align}
	\begin{aligned}
	(E-\om Q)(\vec{u}, \pa_t \vec{u}) ={}& J_{\om}(\vec{u}) \\
	&{}+ \Lebn{\pa_t u_1 - i\om u_1}{2}^2 + \frac{1}{2}\Lebn{\pa_t u_2 - 2i\om u_2}{2}^2. 
	\end{aligned}
	\label{fun:3}
\end{align}

\subsection{Mass-supercritical or mass-critical case}

In view of \eqref{fun:3}, let us define the set
\begin{align*}
	\m{R}_{\om}^1 = \{(\vec{u}, \vec{v}) \in \m{H}^1 \times \m{H}^0;\; (E-\om Q)(\vec{u}, \vec{v}) < d_{\om}^1,\; K( \vec{u} )<0 \}. 
%\label{set:1}
\end{align*}
We see the following:
\begin{lemma} \label{lem:2}
Let $N=4$, $5$. Fix $\om \in \R$ with $|\om| < \min\(1, \ka/2 \)$. Then
\begin{enumerate}
\renewcommand{\labelenumi}{(\roman{enumi})}
\item $J_{\om}(\vec{u}) - \frac{1}{N} K(\vec{u}) > d_{\om}^1$ for all $\vec{u} \in \m{H}^1(\R^N)$ satisfying $K(\vec{u}) <0$. 
\item $\l \( \vec{\f}_{\om}, i\om \f_{1, \om},  2i\om \f_{2, \om} \) \in \m{R}^1_{\om}$ for any $\l>0$ if $\vec{\f}_{\om} \in \m{G}_{\om}$. 
\end{enumerate}
\end{lemma}
\begin{proof}
To show (i), set
\begin{align*}
	J_{\om}^1(\vec{u}) = J_{\om}(\vec{u}) -\frac1{N}K(\vec{u}).
\end{align*}
A calculation shows that
\begin{align*}
	J_{\om}^1(u_1, u_2) ={}& \(\frac12 - \frac2{N} \) L(\vec{u}) + \frac12 M_{\om}(\vec{u}).
%	\label{lem:21}
\end{align*} 
Note that $J_{\om}^1(\vec{u}) \ge 0$ from $N \ge 4$. By $K(\vec{u}) <0$, we have $\vec{u} \neq \vec{0}$ and $P(\vec{u}) >0$. Also, it turns out that for any $\l>0$,
\begin{align*}
		K(\l \vec{u}) ={}& \l^{2} \( 2 L(\vec{u}) - \frac{N\l}2 P(\vec{u}) \).
\end{align*}
This allows us to take $\l_0 \in (0,1)$ such that $K(\l_0 \vec{u}) =0$. Recalling \eqref{mini:1}, from the fact $\pa_{\l}J_{\om}^1(\l \vec{u})<0$ for all $\l>0$, we see that
\[
	d_{\om}^1 \le J_{\om}(\l_0 \vec{u}) = J^1_{\om}(\l_0 \vec{u}) < J^1_{\om}(\vec{u}),
\]
which implies (i).

%By Theorem \ref{gs:thm1}, we obtain (2).

We shall show (ii). By Proposition \ref{gs:t0}, $J_{\om}(\vec{\f}_{\om}) = d_{\om}^1$. 
%We here remark that
%from the integration by part, \eqref{snlkg} gives us 
%\begin{align}
%	\begin{aligned}
%	&{}\Lebn{\n \f_{1,\om}}{2}^2 + \frac12\Lebn{\n \f_{2,\om}}{2}^2 \\
%	+{}& (m_1^2 - \om^2)\Lebn{\f_{1, \om}}{2}^2 + \frac{m_2^2 - 4\om^2}{2}\Lebn{\f_{2, \om}}{2}^2 = \frac32 \int_{\R^N} \f_{1,\om}^2 \overline{\f_{2,\om}} dx.
%	\end{aligned}
%	\label{lgs:2}
%\end{align}
A computation shows
\begin{align*}
	J_{\om}(\l \vec{\f}_{\om}) ={}& \frac{\l^2}{2} \( L(\vec{\f}_{\om}) + M_{\om} (\vec{\f}_{\om}) - \l P(\vec{\f}_{\om}) \).
\end{align*}
By means of Lemma \ref{lgs:0} (ii), we have $\pa_{\l} J_{\om}(\l \f_{1,\om}, \l \f_{2,\om})|_{\l=1} = 0$, which also yields $P(\vec{\f}_{\om}) >0$.
The fact and Lemma \ref{lgs:0} (ii) tell us $\pa_{\l}J_{\om}(\l \vec{\f}_{\om}) < 0$ for any $\l \ge 1$. Hence, combining these above with \eqref{fun:3}, we see that
\[
	(E-\om Q)(\l(\vec{\f}_{\om}, i\om\f_{1, \om}, 2i\om\f_{1, \om})) = J_{\om}(\l \vec{\f}_{\om} ) < d_{\om}^1
\]
for any $\l >1$. Further, it is deduced that
\begin{align*}
	&{}K(\l \vec{\f}_{\om} ) =  2\l^{2} L(\vec{\f}_{\om})  - \frac{N\l^3}2 P(\vec{\f}_{\om}).
\end{align*}
From $N \ge 4$, we obtain $\pa_{\l}K(\l \vec{\f}_{\om}) <0$ for any $\l>1$. Since $K(\l \vec{\f}_{\om})|_{\l=1} =0$, 
it holds that $K(\l \vec{\f}_{\om})<0$ for any $\l>1$. Thus, we conclude (ii). 
\end{proof}

\begin{lemma} \label{lem:3}
Let $N=4$, $5$. Set $\om \in \R$ with $|\om| < \min \(1, \ka/2\)$. If $(\vec{\v}, \vec{\g}) \in \m{R}_{\om}^1$, then the solution $\vec{u} \in C([0,T_{\max}), \m{H}^1)$ of \eqref{nlkg} with $\vec{u}(0) = \vec{\v}$ and $\pa_t \vec{u}(0) = \vec{\g}$ satisfies 
\begin{align*}
%	K(\vec{u}(t)) <{}& 0, \label{lem:30} \\
	-\frac{1}{N}K(\vec{u}(t)) >{}& d_{\om}^1 - (E-\om Q)(\vec{\v}, \vec{\g}) %\label{lem:31}
\end{align*}
for any $t \in [0,T_{\max})$.
\end{lemma}
\begin{proof}
Let us begin with the proof of the fact $K(\vec{u}(t))<0$ for any $t \in [0,T_{\max})$ by a contradiction. Suppose that there exists $t_1 \in (0, T_{\max})$ such that $K(\vec{u}(t_1)) =0$ and $K(\vec{u}(t)) < 0$ for any $t \in [0,t_1)$. 
Lemma \ref{lem:2} (i) gives us
\begin{align*}
	\(\frac12 - \frac2{N} \) L(\vec{u}) + \frac12 M_{\om}(\vec{u}) > d_{\om}^1 >0
\end{align*}
for any $t \in [0,t_1)$, Taking $t \to t_1$, we have $\vec{u}(t_1) \neq \vec{0}$. By \eqref{mini:1}, one sees $J_{\om}(\vec{u}(t_1)) \ge d_{\om}^1$. On the other hand, it follows from \eqref{fun:3}, $(\vec{\v}, \vec{\g}) \in \m{R}_{\om}^1$ and the fact  $E$ and $Q$ are conserved for all $t$ that
\begin{align*}
	J_{\om}(\vec{u}(t_1)) \le{}& (E-\om Q)(\vec{u}(t_1)) = (E-\om Q)(\vec{\v}, \vec{\g}) < d_{\om}^1.
\end{align*}
This yields a contradiction. Thus, $K(\vec{u}(t))<0$ for any $t \in [0,T_{\max})$.
Together with the above and the fact $E$ and $Q$ are conserved for all $t$, combining Lemma \ref{lem:2} (i) with \eqref{fun:3}, 
it turns out that 
\begin{align*}
	- \frac{1}{N} K(\vec{u}(t)) >{}& d_{\om}^1 - J_{\om}(\vec{u}) \ge d_{\om}^1 - (E-\om Q)(\vec{\v}, \vec{\g}) 
\end{align*}
for any $t \in [0,T_{\max})$, which completes the proof.
\end{proof}

\subsection{Mass-subcritical case}

We next define the set  
\begin{align*}
	\m{R}_{\om}^0 = \{(\vec{u}, \vec{v}) \in \m{H}^1 \times \m{H}^0;\; (E-\om Q)(\vec{u}, \vec{v}) < d_{\om}^0, \quad K_{\om}^0(\vec{u})<0 \}. %\label{set:2}
\end{align*}

We have the following:
\begin{lemma} \label{lem:4}
Let $N=2$, $3$ Fix $\om \in \R$ with $|\om| < \min\(1, \ka/2 \)$. Then
\begin{enumerate}
\renewcommand{\labelenumi}{(\roman{enumi})}
\item $\ds \frac{M_{\om}(\vec{u})}{\a+2}  > d_{\om}^0$ for all $\vec{u} \in \m{H}^1(\R^N)$ satisfying $K_{\om}^0(\vec{u}) <0$. 
%\item The minimization problem \eqref{mini:2} is attained at the ground state $\vec{\f}_{\om}$ of \eqref{snlkg}. \label{lem:4b}
\item $\l \( \vec{\f}_{\om}, i\om \f_{1, \om},  2i\om \f_{2, \om} \) \in \m{R}^0_{\om}$ for any $\l>0$ if $\vec{\f}_{\om} \in \m{G}_{\om}$.
\end{enumerate}
\end{lemma}

%Remark that it holds that
%\begin{align}
%	&{}J_{\om}(\vec{u}) - \frac{K_{\om}^0(\vec{u})}{2(\a+2)} = \frac{M_{\om}(\vec{u})}{\a+2}  
%	\label{lem:52} \\
%	d_{\om}^0 ={}&  \inf \left\{ \left.\frac{M_{\om}(\vec{u})}{\a+2}  \right.|\ \vec{u} \in (H^1(\R^{N}))^2 \setminus \vec{0},\ K_{\om}^0(\vec{u}) = 0 \right\}. 
%	\label{lem:53}
%\end{align}

\begin{proof}[Proof of Lemma \ref{lem:4}]
We first compute
\begin{align*}
	J_{\om}(\vec{u}) -\frac{K_{\om}^0(\vec{u})}{2(\a+2)} = \frac{M_{\om}(\vec{u})}{\a+2}. %\label{lem:41}
\end{align*} 
$K_{\om}^0(\vec{u}) <0$ implies $\vec{u} \neq \vec{0}$ and $P(\vec{u}) >0$. Also, we deduce that for any $\l>0$,
\begin{align}
		K_{\om}^0(\l \vec{u}) = \l^{2} \a M_{\om}(\vec{u}) + \l^2 (\a+2) L(\vec{u}) - \l^3  (\a+2) P(\vec{u}).
	\label{lem:42}
\end{align}
This allows us to take $\l_1 \in (0,1)$ such that $K_{\om}^0(\l_1 \vec{u}) =0$. Recalling \eqref{mini:2}, one sees that
\begin{align*}
	d_{\om}^0 \le \frac{M_{\om}(\l_1 \vec{u})}{\a+2}  < \frac{M_{\om}(\vec{u})}{\a+2},
\end{align*}
which implies (i).

We shall show (ii). In the same manner as in the proof of Lemma \ref{lem:2} (ii), we have  
\[
	(E-\om Q)\( \l \( \vec{\f}_{\om}, i\om \f_{1, \om},  2i\om \f_{2, \om} \) \) < d_{\om}^0
\]
for any $\l >1$.
Further, it follows from \eqref{lem:42} that 
\[
	K_{\om}^{0}(\l \vec{\f}_{\om}) < 0
\]
for any $\l >1$, which yields (ii).
\end{proof}

\begin{lemma} \label{lem:5}
Let $N=2$, $3$. Set $\om \in \R$ with $|\om| < \min \(1, \ka/2\)$. If $(\vec{\v}, \vec{\g}) \in \m{R}_{\om}^0$, then the solution $\vec{u} \in C([0,T_{\max}), \m{H}^1)$ of \eqref{nlkg} with $\vec{u}(0) = \vec{\v}$ and $\vec{u}(0) = \vec{\g}$ satisfies 
\begin{align}
	K_{\om}^0(\vec{u}(t)) <{}& 0, \label{lem:50} \\
	\frac{M_{\om}(\vec{u}(t))}{\a+2} >{}& d_{\om}^0 \label{lem:51}
\end{align}
for any $t \in [0,T_{\max})$.
\end{lemma}
\begin{proof}
As in the proof of Lemma \ref{lem:3}, let us get started with the proof of the fact $K_{\om}^0(\vec{u}(t))<0$ for any $t \in [0,T_{\max})$ by a contradiction. Suppose that there exists $t_1 \in (0, T_{\max})$ such that $K_{\om}^0(\vec{u}(t_1)) =0$ and $K_{\om}^0(\vec{u}(t)) < 0$ for any $t \in [0, t_1)$. 
Lemma \ref{lem:4} (i) gives us
\begin{align*}
	&{} \frac{M_{\om}(\vec{u})}{\a+2}  > d_{\om}^0 >0
\end{align*}
for any $t \in [0,t_1)$, Taking $t \to t_1$, we have $\vec{u}(t_1) \neq \vec{0}$. By \eqref{mini:2}, one sees $J_{\om}(\vec{u}(t_1)) \ge d_{\om}^0$. On the other hand, it follows from \eqref{fun:3}, $(\vec{\v}, \vec{\g}) \in \m{R}_{\om}^0$ and the fact  $E$ and $Q$ are conserved for all $t$ that
\begin{align*}
	J_{\om}(\vec{u}(t_1)) \le{}& (E-\om Q)(\vec{u}(t_1), \pa_t \vec{u}(t_1)) \\
	={}& (E-\om Q)(\vec{\v}, \vec{\g}) < d_{\om}^0.
\end{align*}
This yields a contradiction. Thus, we have $K_{\om}^0(\vec{u}(t))<0$ for any $t \in [0,T_{\max})$. Thus, Lemma \ref{lem:4} (ii) gives us \eqref{lem:51}.
\end{proof}

\section{Proof of the main results} \label{sec:6}

This section consists of the proof of the main results.
Before going to the proof of the main results, we shall handle the following lemma concerning the uniform boundedness of global solutions to \eqref{nlkg}:
\begin{lemma} \label{lem:uni}
Let $\vec{u} \in C([0, \I), \m{H}^1)$ be a solution to \eqref{nlkg}. Then 
\begin{align}
	&{}\sup_{t \ge 0} \Lebn{\vec{u}(t)}{2}  < \I, \label{lem:01} \\
	&{}\sup_{t \ge 0}\int_t^{t+1} \norm{(\vec{u}, \pa_t \vec{u})(t)}_{H^1 \times L^2}^2 ds < \I. \label{lem:02}
\end{align}
\end{lemma}
\begin{proof}
The proof is carried out as in Proposition 3.1 and Lemma 3.5 in \cite{C1}.
In detail, we refer readers to the Appendix \ref{app:a}. 
\end{proof}

\subsection{Proof of Theorem \ref{thm:1}}

Let us start to prove the mass-supercritical and critical case $N=4$, $5$.

\begin{proof}[Proof of Theorem \ref{thm:1} in $N=4$, $5$]
The strategy of the proof is based on \cite[Theorem A]{OT}. 
%due to Kenji Nakanishi. 
%However, we shall give the proof for self-containedness. 
For any $\e>0$, we can take $\l=\l( \vec{\f}_{\om}, \om)>1$ satisfying 
\begin{align*}
	(\l-1) \( \norm{(\f_{1, \om}, i\om \f_{1, \om})}_{H^1 \times L^2} + \norm{(\f_{2, \om}, 2i\om \f_{2, \om})}_{H^1 \times L^2} \) < \e. %\label{pm:1}
\end{align*}
Put
\begin{align*}
	\d = \frac{N}{2} \left\{ d_{\om}^1 - (E-\om Q)\( \l ( \vec{\f}_{\om}, i\om\f_{1, \om}, 2i\om\f_{1, \om}) \) \right\}. %\label{pm:2}
\end{align*}
Lemma \ref{lem:2} (ii) yields $\d>0$. 
Here, \eqref{lem:02} allows us to exist $J>0$ such that
\begin{align}
	\int_T^{T+1} \( \int_{\R^N} |\pa_t u_j(t,x)|^2 + |\n u_j(t,x)|^2 + |u_j(t,x)|^2 dx \) dt \le J \label{thm:21}
\end{align}
for $j=1$, $2$ and all $T>0$.
we deduce from the mean value theorem that for any $i \in \Z_{+}$, there exists $T_i \in [i-1, i]$ such that
\begin{align*}
	\int_{\R^N} |\pa_t u_j(T_i)|^2 + |\n u_j(T_i)|^2 + |u_j(T_i)|^2 dx  \le J. %\label{thm:22}
\end{align*}
Let us prove by a contradiction. 
Assume that there exists a global solution $\vec{u} \in C([0,\I), \m{H}^1)$ of \eqref{nlkg} with $\vec{u}(0) = \l \vec{\f}_{\om}$ and $\pa_t \vec{u}(0) = \l(i\om \f_{1, \om}, 2i\om \f_{2, \om})$. 
Since $\vec{u}$ is radially symmetric with respect to $x$ for any $t \ge 0$, 
we can define $I_{\rho}^1(\vec{u}, \pa_t \vec{u})$ by \eqref{qu:1}, so that one has \eqref{ineq:1}. 
By using \eqref{ineq:1}, Lemma \ref{lem:2} and Lemma \ref{lem:3}, we have
\begin{align}
	&{}\frac{d}{dt}I_{\rho}^1(t) \ge 2\d -R_{\rho}(t) \label{thm:23}
\end{align}
for any $t \ge 0$ and all $\rho >0$, where $I_{\rho}^1(t) = I_{\rho}^1(\vec{u}(t), \pa_t \vec{u}(t))$ is defined by \eqref{qu:1} and
\begin{align}
	R_{\rho}(t) = \frac12 \re\int_{\{|x| \ge \rho\}} u_1^2(t) \overline{u_2(t)} dx + \frac{C_0}{\rho}M(\vec{u}(t)). \label{thm1s:2}
\end{align}
Integrating the both side of \eqref{thm:23} for $t$, one sees  from $T_{i+2} -T_i \ge 1$ that
\begin{align}
	I_{\rho}^1(T_{i+2}) - I_{\rho}^1(T_i) \ge 2\d - \int_{T_i}^{T_{i+2}} R_{\rho}(t) dt \label{thm:24}
\end{align}
for all $i \in \Z_+$. 
Let us here show that there exists a constant $C_1>0$ not depending on $i$ such that
\begin{align}
	\int_{T_i}^{T_{i+2}} \re\int_{\{|x| \ge \rho\}} u_1^2(t) \overline{u_2(t)} dx dt \le C_1 \rho^{-\frac{N-1}{2}}J^{\frac32}. \label{thm:25}
\end{align}
To this end, let $\chi(t, r) \in C^{\I}_0(\R^2)$ satisfy $\chi(t,r) =1$ if $|t| \le 2$ and $|r| \ge 1$, $\chi(t,r) =0$ if $|t| \ge 4$ or $|r| \le 1/2$, and $0 \le \chi(t,r) \le 1$. Set $\vec{v}(t,r) = \chi(t-T, r/(2^k \rho))\vec{u}(t, |r|)$ for any $\rho>1$, all $T>4$ and $k \in \Z_{\ge 0}$. Since $\vec{u}$ is a radially symmetric function, by means of \eqref{thm:21}, we estimate
\begin{align*}
	&{}\int_{\R^2} |\pa_t v_{j}|^2 + |\pa_r v_j|^2 + |v_j|^2 dr dt \\
%	\le{}& C2^k \rho \int_{T-4}^{T+4} \int_{|\wt{r}| \ge 2^{k-1}\rho} |\pa_t u_{j}(2^k \rho|\wt{r}|)|^2 + |\pa_r u_j(2^k \rho|\wt{r}|)|^2 + |u_j(2^k \rho|\wt{r}|)|^2 d\wt{r} dt \\
	\le{}& C2^{k} \rho  \\
	&{}\times \int_{T-4}^{T+4} \int_{1/2}^{\I} \( |\pa_t u_{j}(2^k \rho \wt{r})|^2 + |\pa_r u_j(2^k \rho \wt{r})|^2 + |u_j(2^k \rho \wt{r})|^2 \) (2\wt{r})^{N-1} d\wt{r} dt \\
	\le{}& C(2^k \rho)^{1-N} \int_{T-4}^{T+4} \( \int_{\R^N} |\pa_t u_{j}|^2 + |\n u_j|^2 + |u_j|^2  dx \) dt \\
	\le{}& 8C(2^k \rho)^{1-N} J
\end{align*}
for any $\rho>1$, all $T>4$ and $k \in \Z_{\ge 0}$. Together with the above, it follows from the Young inequality and Sobolev embedding that
\begin{align*}
	&{}\int_{T-2}^{T+2} \re\int_{\{|x| \ge \rho\}} u_1^2 \overline{u_2} dx dt \\
	\le{}& C \sum_{j=1}^2 \int_{T-2}^{T+2} \int_{\{|x| \ge \rho\}} |u_j|^3 dx dt \\
%	\le{}& C \sum_{j=1}^2 \int_{T-2}^{T+2} \int_{r \ge \rho} |u_j(r)|^3 r^{N-1} dr dt \\
	\le{}& C \sum_{j=1}^2 \int_{T-2}^{T+2} \sum_{k=0}^{\I} \int_{2^{k} \rho}^{2^{k+1} \rho} |u_j(r)|^3 r^{N-1} dr dt \\
%	\le{}& C \sum_{j=1}^2 \int_{T-2}^{T+2} \sum_{k=0}^{\I} (2^{k+1} \rho)^{N-1} \int_{2^{k} \rho}^{\I} |u_j(r)|^3 dr dt \\
	\le{}& C \sum_{j=1}^2 \sum_{k=0}^{\I} (2^{k} \rho)^{N-1} \int_{\R^2} |v_j(r)|^3 dr dt \\
	\le{}& C\sum_{j=1}^2 \sum_{k=0}^{\I} (2^{k} \rho)^{N-1} \( \int_{\R^2} |\pa_t v_{j}|^2 + |\pa_r v_j|^2 + |v_j|^2 dr dt \) ^{3/2} \\
%	\le{}& C \sum_{k=0}^{\I} (2^{k} \rho)^{-\frac12(N-1)} M^{3/2} \\
	\le{}& C \rho^{-\frac{N-1}2} J^{\frac32} \sum_{k=0}^{\I} 2^{-\frac{k(N-1)}{2}},
\end{align*}
which yields \eqref{thm:25}. In view of \eqref{lem:01}, this allows us to exist $\rho_0 >0$ such that
\begin{align*}
	\int_{T_i}^{T_{i+2}} R_{\rho_0}(t) dt < \d
\end{align*}
for any $i \ge 4$.
We see from the above and \eqref{thm:24} that
\begin{align*}
	I_{\rho_0}^1(T_{i+2}) - I_{\rho_0}^1(T_i) \ge \d
\end{align*}
for any $i \ge 4$, which contradicts that for any $i \ge 1$,
\[
	I^1_{\rho_0}(T_i) \le C\rho_0 \norm{\( \vec{u}(T_i), \pa_t \vec{u}(T_i)\)}_{H^1 \times L^2}^2 \le C\rho_0 J.
\]
Hence, the solution blows up at finite time. This completes the proof in $N=4$, $5$.

\end{proof}

Before proving the mass-subcritical case $N=2$, $3$, We note that the following identity holds:
\begin{align}
	\begin{aligned}
	&{} H(\vec{u}, \pa_t \vec{u}) \\
	={}& 2(\a+2)(E-\om Q)(\vec{u}_1, \vec{u}_2) - 2\om \a Q(\vec{u}_1, \vec{u}_2) \\
	&{} -2(\a+1)\( \Lebn{\pa_t u_1 - i\om u_1}{2}^2 + \frac{1}{2}\Lebn{\pa_t u_2 - 2i\om u_2}{2}^2 \) \\
	&{} -2(1 -(\a+1)\om^2)\Lebn{u_1}{2}^2 - (\ka^2 - 4(\a+1)\om^2)\Lebn{u_2}{2}^2. 
	\end{aligned}
	\label{fun:6}
\end{align}

\begin{proof}[Proof of Theorem \ref{thm:1} in $N=2$, $3$]
For any $\e>0$, there exists $\l=\l( \vec{\f}_{\om}, \om)>1$ satisfying 
\begin{align*}
	(\l-1) \( \norm{(\f_{1, \om}, i\om \f_{1, \om})}_{H^1 \times L^2} + \norm{(\f_{2, \om}, 2i\om \f_{2, \om})}_{H^1 \times L^2} \) < \e. %\label{pm:1}
\end{align*}
Set
\begin{align}
	\begin{aligned}
	\d_1 ={}& 2(\a+2)\( d_{\om}^0 - (E-\om Q)(\vec{u}(0), \pa_t \vec{u}(0)) \), \\
	\d_2 ={}& \a \( \om Q(\vec{u}(0), \pa_t \vec{u}(0)) - \frac{\om^2(\a+2)d_{\om}^0}{1-\om^2} \).
	\end{aligned}
	\label{delta:1}
\end{align}
Lemma \ref{lem:4} (ii) yields $\d_1>0$. Further, form $\vec{\f}_{\om} \in \m{G}_{\om}$ and $\ka =2$, we see that
\begin{align*}
	d_{\om}^0 ={}& \frac{M_{\om}(\vec{\f}_{\om})}{\a+2} = \frac{1 -\om^2}{\a+2} \( \Lebn{\f_{1, \om}}{2}^2 + 2\Lebn{\f_{2, \om}}{2}^2\),
\end{align*} 
which implies
\begin{align*}
	\frac{\om^2(\a+2)d_{\om}^0}{1-\om^2} ={}& \om^2 \( \Lebn{\f_{1, \om}}{2}^2 + 2\Lebn{\f_{2, \om}}{2}^2\) \\
	 <{}& \l^2 \om^2 \( \Lebn{\f_{1, \om}}{2}^2 + 2\Lebn{\f_{2, \om}}{2}^2\) = \om Q(\vec{u}(0), \pa_t \vec{u}(0))
\end{align*}
for any $\l>1$. Thus $\d_2>0$. We denote $\d=\d_1+\d_2$.
%As in the proof of the case $N = 4$, 
Let us prove by a contradiction as in $N \ge 4$.
Assume that there exists a global solution $\vec{u} \in C([0,\I), \m{H}^1)$ of \eqref{nlkg} with $\vec{u}(0) = \l \vec{\f}_{\om}$
and $\pa_t \vec{u}(0) = \l(i\om \f_{1, \om}, 2i\om \f_{2, \om})$.
% and the solution satisfies 
%\begin{align*}
%	M_1 = \sup_{t \ge 0}\norm{\( \vec{u}(t), \pa_t \vec{u}(t)\)}_{H^1 \times L^2} < \I.
%\end{align*}
Since $\vec{u}(t)$ is radially symmetric with respect to $x$ for any $t \ge 0$, we can define $I_{\rho}^2(\vec{u}, \pa_t \vec{u})$ by \eqref{qu:2}.
Thus \eqref{ineq:2} is valid. By using \eqref{lem:51}, \eqref{fun:6} and the fact $E (\vec{u}, \pa_t \vec{u})$ and $Q(\vec{u}, \pa_t \vec{u})$ are conserved for $t \ge 0$, together with $\ka=2$, we obtain
\begin{align*}
	-H(\vec{u}, \pa_t \vec{u}) \ge {}& - 2(\a+2)(E-\om Q)(\vec{u}(0), \pa_t \vec{u}(0)) + 2\om \a Q(\vec{u}(0), \pa_t \vec{u}(0)) \\
	&{} +2(1 -(\a+1)\om^2)\Lebn{u_1}{2}^2 + (\ka^2 - 4(\a+1)\om^2)\Lebn{u_2}{2}^2 \\
	\ge {}& -2(\a+2)(E-\om Q)(\vec{u}(0), \pa_t \vec{u}(0)) +2\om \a Q(\vec{u}(0), \pa_t \vec{u}(0)) \\
	&{} + 2(1 - (\a+1)\om^2)\( \Lebn{u_1(t)}{2}^2 + 2 \Lebn{u_2(t)}{2}^2 \) \\
	> {}& 2\d.
\end{align*}
Hence we have
\begin{align*}
	&{}\frac{d}{dt}I_{\rho}^2(\vec{u}(t), \pa_t \vec{u}(t)) \ge 2\d -R_{\rho}(t) %\label{thm1s:1}
\end{align*}
for any $t \ge 0$ and all $\rho >0$, where $R_{\rho}(t)$ is defined by \eqref{thm1s:2}.
%\begin{align*}
%	R_{\rho}(t) = \frac12 \re\int_{\{|x| \ge \rho\}} u_1^2(t) \overline{u_2(t)} dx + \frac{C_0}{\rho}M(\vec{u}(t)).
%\end{align*}
The remaining proof is same way to the case $N \ge 4$, so we omit the proof.

\end{proof}

\subsection{Proof of Theorem \ref{thm:3}}

\begin{proof}[Proof of Theorem \ref{thm:3}]
Let us only prove the case $\om = \om_c$, because the case $\om = - \om_c$ is exactly same.
Let $\vec{\f} \in \m{H}^1$ be a radially symmetric solution to \eqref{snlkg} with $\om =\om_c$. For any $\e>0$, there exists $\l=\l( \vec{\f}, \om)>1$ such that
\begin{align*}
	(\l-1) \( \norm{(\f_{1}, i\om \f_{1})}_{H^1 \times L^2} + \norm{(\f_{2}, 2i\om \f_{2})}_{H^1 \times L^2} \) < \e. %\label{pm:1}
\end{align*}
Set
\begin{align*}
	\d ={}& \a \om_c Q(\l(\vec{\f}, i\om\f_{1}, 2i\om\f_{2})) \\
	&{} - (\a+2)  (E-\om_c Q)(\l( \vec{\f}, i\om\f_{1}, 2i\om\f_{2})).
\end{align*}
We see from Lemma \ref{gs:lem1} (ii) and (iii) that $\pa_{\eta} J_{\om_c}(\eta \vec{\f})|_{\eta =1} =0$ and $\pa_{\eta} J_{\om_c}(\l \vec{\f}) < 0$ for any $\eta >1$. Hence, in view of \eqref{fun:3}, it holds that 
\[
	(E-\om_c Q)(\l(\vec{\f}, i\om\f_{1}, 2i\om\f_{2})) = J_{\om_c}(\l \vec{\f}) < J_{\om_c}(\vec{\f}).
\]
We also estimate 
\begin{align*}
	\om_c Q(\l(\vec{\f}, i\om\f_{1}, 2i\om\f_{2})) ={}& \l^2 \om_c^2 \( \Lebn{\f_{1}}{2}^2 + 2\Lebn{\f_{2}}{2}^2\) \\
	>{}& \om_c^2 \( \Lebn{\f_{1}}{2}^2 + 2\Lebn{\f_{2}}{2}^2\).
\end{align*}
Hence, using $\ka=2$, one has
\begin{align}
	\begin{aligned}
	\d >{}& \a \om_c^2 \( \Lebn{\f_{1}}{2}^2 + 2\Lebn{\f_{2}}{2}^2\) -(\a+2)J_{\om_c}(\vec{\f}) \\
	={}& -\frac12 K_{\om_c}^0(\vec{\f}) - (m_1^2 -(\a+1)\om_c^2) \( \Lebn{\f_{1}}{2}^2 + 2\Lebn{\f_{2}}{2}^2\).
	\end{aligned}
	\label{thm:31}
\end{align}
On the other hand, \eqref{snlkg} leads to $\pa_{u_{1}}J_{\om}(\vec{\f}) = \pa_{u_{2}}J_{\om}(\vec{\f}) = 0$. Indeed, it turns out that
\begin{align*}
	\langle \pa_{u_{1}}J_{\om}(\vec{\f}),  w \rangle ={}& \re \int_{\R^N}(-\D \f_1 +(m_1^2-\om^2)\f_1 -\f_1 \f_2 ) \overline{w} dx =0, \\
	\langle \pa_{u_{2}}J_{\om}(\vec{\f}),  w \rangle ={}& \frac12 \re \int_{\R^N}(-\D \f_2 +(m_2^2-4\om^2)\f_1 -\f_1^2 ) \overline{w} dx =0
\end{align*}
for any $w \in C_{0}^{\I}(\R^N)$. 
We also see from \eqref{snlkg} and the elliptic regularization method that $x \cdot \n \f_j \in H^1$ for $j=1$, $2$ (see \cite[Theorem 8.1.1]{C2}). Therefore it follows that
\begin{align*}
	K_{\om}^0(\vec{\f}) ={}& 2\pa_{\l}J_{\om}(\l^\b \vec{\f}(\l \cdot))|_{\l=1} \\
	={}& \langle \pa_{u_{1}}J_{\om}(\vec{\f}),  \b \f_1 + x \cdot \n \f_1 \rangle + \langle \pa_{u_{2}}J_{\om}(\vec{\f}),  \b \f_2 + x \cdot \n \f_2 \rangle \\
	={}& 0.
\end{align*}
Combining the above fact with $m_1^2 -(\a+1)\om_c^2 =0$, together with \eqref{thm:31}, we have $\d>0$.
As in the proof of Theorem \ref{thm:1}, let us prove by a contradiction. Assume that there exists a global solution $\vec{u} \in C([0,\I), \m{H}^1)$ of \eqref{nlkg} with $\vec{u}(0) = \l \vec{\f}$ and $\pa_t \vec{u}(0) = \l(i\om \f_{1}, 2i\om \f_{2})$.
% and the solution satisfies 
%\begin{align*}
%	M_1 = \sup_{t \ge 0}\norm{\( \vec{u}(t), \pa_t \vec{u}(t)\)}_{H^1 \times L^2} < \I.
%\end{align*}
Since $\vec{u}$ is radially symmetric with respect to $x$ for any $t \ge 0$, we can define $I_{\rho}^2(\vec{u}, \pa_t \vec{u})$ by \eqref{qu:1}, so that one has \eqref{ineq:2}. By using \eqref{lem:51}, \eqref{fun:6} and the fact $E (\vec{u}, \pa_t \vec{u})$ and $Q (\vec{u}, \pa_t \vec{u})$ are conserved for $t \ge 0$, together with $\ka = 2$ and $1 -(\a+1)\om_c^2 =0$, we see that
\begin{align*}
	&{}-H(\vec{u}(t), \pa_t \vec{u}(t)) \\
	\ge {}& - 2(\a+2)(E-\om_c Q)(\vec{u}(0), \pa_t \vec{u}(0)) + 2\om_c \a Q(\vec{u}(0), \pa_t \vec{u}(0)) \\
	&{} +2(1 -(\a+1)\om_c^2)\Lebn{u_1}{2}^2 + (\ka^2 - 4(\a+1)\om_c^2)\Lebn{u_2}{2}^2 \\
	\ge {}& 2\om_c \a Q(\vec{u}(0), \pa_t \vec{u}(0)) -2(\a+2)(E-\om_c Q)(\vec{u}(0), \pa_t \vec{u}(0)) \\
	&{} + 2(1 - (\a+1)\om_c^2)\( \Lebn{u_1(t)}{2}^2 + 2 \Lebn{u_2(t)}{2}^2 \) \\
	= {}& 2\d.
\end{align*}
Hence it is obtained that
\begin{align*}
	&{}\frac{d}{dt}I_{\rho}^2(\vec{u}(t), \pa_t \vec{u}(t)) \ge 2\d -R_{\rho}(t) 
\end{align*}
for any $t \ge 0$ and all $\rho >0$, where $R_{\rho}(t)$ is defined by \eqref{thm1s:2}.
The remainder of the proof is exactly same to that of Theorem \ref{thm:1}, so we omit the proof.

%
%
%Arguing as in Theorem \ref{thm:1}, there exists $\rho_0 >0$ such that
%\begin{align*}
%	\sup_{t \ge 0} \( \frac12 \re\int_{\{|x| \ge \rho_0\}} u_1^2(t) \overline{u_2(t)} dx + \frac{C_0}{\rho_0^2}M(\vec{u}(t)) \) < \d.
%\end{align*}
%Combining the above with \eqref{ineq:2} gives us
%\begin{align*}
%	&{}\frac{d}{dt}I_{\rho_0}^2(\vec{u}(t), \pa_t \vec{u}(t)) \ge -H(\vec{u}(t), \pa_t \vec{u}(t)) - \d
%\end{align*}  
%for any $t \ge 0$. 
%Hence, one has $\lim_{t \to \I}I_{\rho}^2(\vec{u}(t), \pa_t \vec{u}(t)) = \I$.
%On the other hand, there exists $C = C(\rho_0)>0$ such that
%\begin{align*}
%	I_{\rho_0}^2(\vec{u}(t), \pa_t \vec{u}(t)) \le C\norm{\( \vec{u}(t), \pa_t \vec{u}(t)\)}_{H^1 \times L^2}^2 \le CM_1^2
%\end{align*}
%for any $t \ge 0$. This yields a contradiction. By \eqref{lem:50}, we have $P(\vec{u}(t)) <0$ for any $t \ge 0$. From Proposition \ref{lem:0}, the solution $\vec{u}(t)$ satisfies 
%\[
%	M_1 = \sup_{t \ge 0} \norm{(\vec{u}, \pa_t \vec{u})(t)}_{H^1 \times L^2} < \I.
%\]
%Hence, the solution blows up at finite time. This proof is completed.

\end{proof}

\subsection{Application of the uniform boundedness of global solutions}

In the mass-subcritical or mass-critical case $2 \le N \le 4$, we establish the following: 
\begin{proposition} \label{lem:0}
Let $N =2$, $3$, $4$. If $\vec{u} \in C([0,\I), \m{H}^1)$ is a global solution of \eqref{nlkg} satisfying $P(\vec{u}(t))  >0$ for any $t \ge 0$, then 
\[
	\sup_{t \ge 0} \norm{(\vec{u}, \pa_t \vec{u})(t)}_{H^1 \times L^2} < \I
\]
holds. 
%where 
%\[
%	\norm{(\vec{u}, \pa_t \vec{u})}_{H^1 \times L^2} = \norm{(u_1, \pa_t u_1)}_{H^1 \times L^2} + \norm{(u_2, \pa_t u_2)}_{H^1 \times L^2}.
%\]
\end{proposition}
\begin{proof}[Proof of Proposition \ref{lem:0}]
By the Young inequality $ab \le \e a^p + \e^{-\frac{1}{p-1}} b^q$ with $p^{-1} + q^{-1} =1$, it holds that  
\[
	P(\vec{u}) \le C\( \Lebn{u_1}{3}^3 + \Lebn{u_2}{3}^3 \).
\]
Combining $P(\vec{u}) > 0$ with \eqref{lem:02} and the conservation of the energy, we have
\begin{align*}
	\sup_{t \ge 0}\int_t^{t+1} \( \Lebn{u_1(s)}{3}^3 + \Lebn{u_2(s)}{3}^3 \) ds < \I. %\label{lem:03} 
\end{align*}
Hereafter, we follow the argument by \cite{MZ}. In the same way as in \cite[Lemma 2.1]{OT}, it turns out that
\begin{align}
	A := \sup_{t \ge 0} \Lebn{\vec{u}(t)}{5/2} < \I. \label{lem:04}
\end{align} 
Further, the Gagliardo-Nirenberg inequality gives us
\begin{align}
	\Lebn{u_j}{3} \le C\Lebn{u_j}{5/2}^{1-\t} \Lebn{\n u_j}{2}^{\t} \label{lem:05}
\end{align}
for $j=1$, $2$, where $\t = \frac{2N}{3(10-N)}$. Together with \eqref{lem:04} and \eqref{lem:05}, by using the Young inequality again, for any $l >0$, we establish 
\begin{align}
	\frac{1}{l} P(\vec{u}) \le \frac12 L(\vec{u}) + C_1 \label{lem:06}
\end{align}
for some constant $C_1$ depending on $A$ and $L$. Here set $b = \min (1, \ka)$. When $b <1$, by the conservation of the energy, one obtains
\begin{align*}
	&{}\norm{(u_1, \pa_t u_1)(t)}_{H^1 \times L^2}^2 + \frac12 \norm{(u_2, \pa_t u_2)(t)}_{H^1 \times L^2}^2 \\
	\le{}& \frac{4}{b^2}E(\vec{u}(0), \pa_t \vec{u}(0)) + \frac{2}{b^2} P(\vec{u}).
\end{align*}
we see from \eqref{lem:06} that
\begin{align*}
	&{}\norm{(u_1, \pa_t u_1)(t)}_{H^1 \times L^2}^2 + \frac12 \norm{(u_2, \pa_t u_2)(t)}_{H^1 \times L^2}^2 \\
	\le{}& \frac{8}{b^2}E(\vec{u}(0), \pa_t \vec{u}(0)) + 2C_1,
\end{align*}
which implies $\sup_{t \ge 0} \norm{(\vec{u}, \pa_t \vec{u})(t)}_{H^1 \times L^2} < \I$. The case $m \ge 1$ is easier. Thus, the proof is completed.
\end{proof}

Combining Proposition \ref{lem:0} with the Strauss decay estimate 
\begin{align}
	\norm{f}_{L^{\I}(|x| \ge \rho)} \le C\rho^{-\frac{N-1}{2}}\norm{f}_{H^{1}_{\rm{rad}}} \label{strauss}
\end{align}
for any $\rho >0$ (see \cite{St}), we establish an alternative proof of Theorem \ref{thm:1} in $N \le 4$ and Theorem \ref{thm:3} which is similar to that of Theorem 1 and Theorem 4 in \cite{OT}, respectively. For self-containedness, we only give the proof of Theorem \ref{thm:1} in $N=2$, $3$.

\begin{proof}[Alternative proof of Theorem \ref{thm:1} in $N=2$, $3$]
Set $\d = \d_1 +\d_2$ as in \eqref{delta:1}. For any $\e>0$, there exists $\l=\l( \vec{\f}_{\om}, \om)>1$ such that
\begin{align*}
	(\l-1) \( \norm{(\f_{1, \om}, i\om \f_{1, \om})}_{H^1 \times L^2} + \norm{(\f_{2, \om}, 2i\om \f_{2, \om})}_{H^1 \times L^2} \) < \e. %\label{pm:1}
\end{align*}
Let us prove by a contradiction.
Assume that there exists a global solution $\vec{u} \in C([0,\I), \m{H}^1)$ of \eqref{nlkg} with $\vec{u}(0) = \l \vec{\f}_{\om}$ and $\pa_t \vec{u}(0) = \l(i\om \f_{1, \om}, 2i\om \f_{2, \om})$ and the solution satisfies 
\begin{align*}
	M_1 = \sup_{t \ge 0}\norm{\( \vec{u}(t), \pa_t \vec{u}(t)\)}_{H^1 \times L^2} < \I.
\end{align*}
Since $\vec{u}(t)$ is radially symmetric with respect to $x$ for any $t \ge 0$, we can define $I_{\rho}^2(\vec{u}, \pa_t \vec{u})$ by \eqref{qu:2}.
Thus \eqref{ineq:2} is valid. 
Thanks to the Strauss decay estimate \eqref{strauss}, it is deduced that
\begin{align*}
	\re \int_{|x| \ge \rho} u_1^2(t) \overline{u_2(t)} dx \le{}& \norm{u_2(t)}_{L^{\I}(|x| \ge \rho)} \Lebn{u_1(t)}{2}^2 \\
	\le{}& C\rho^{-\frac{N-1}{2}} \Sobn{u_2(t)}{1}\Sobn{u_1(t)}{1}^2 \\
	\le{}& CM_1^3 \rho^{-\frac{N-1}{2}}
\end{align*}
for all $t \ge 0$ and any $\rho >0$. Combining $N \ge 2$ with \eqref{lem:01}, this allows us to exist $\rho_0 >0$ such that
\begin{align*}
	\sup_{t \ge 0} \( \frac12 \re\int_{|x| \ge \rho_0} u_1^2(t) \overline{u_2(t)} dx + \frac{C_0}{\rho_0^2} M(\vec{u}(t)) \) < \d.
\end{align*}
%Arguing as in the case $N \ge 4$, there exists $\rho_0 >0$ such that
%\begin{align*}
%	\sup_{t \ge 0} \( \frac12 \re\int_{|x| \ge \rho_0} u_1^2(t) \overline{u_2(t)} dx + \frac{C_0}{\rho_0^2} M(\vec{u}(t)) \) < \d.
%\end{align*}
Combining the above with \eqref{ineq:2} gives us
\begin{align*}
	&{}\frac{d}{dt}I_{\rho_0}^2(\vec{u}(t), \pa_t \vec{u}(t)) \ge -H(\vec{u}(t), \pa_t \vec{u}(t)) - \d
\end{align*}  
for any $t \ge 0$. Similarly to the proof of Theorem \ref{thm:1} in $N = 2$, $3$, we obtain $-H(\vec{u}, \pa_t \vec{u}) > 2\d$.
Hence, one has $\lim_{t \to \I}I_{\rho}^2(\vec{u}(t), \pa_t \vec{u}(t)) = \I$.
On the other hand, there exists $C = C(\rho_0)>0$ such that
\begin{align*}
	I_{\rho_0}^2(\vec{u}(t), \pa_t \vec{u}(t)) \le C\norm{\( \vec{u}(t), \pa_t \vec{u}(t)\)}_{H^1 \times L^2}^2 \le CM_1^2
\end{align*}
for any $t \ge 0$. This yields a contradiction. Since \eqref{lem:50} tells us $P(\vec{u}(t)) >0$ for any $t \ge 0$, we see from Proposition \ref{lem:0} that the solution $\vec{u}(t)$ satisfies 
\[
	M_1 = \sup_{t \ge 0} \norm{(\vec{u}, \pa_t \vec{u})(t)}_{H^1 \times L^2} < \I.
\]
Therefore, the solution blows up at finite time. This completes the proof.
\end{proof}

\appendix
\section{Proof of the uniform boundedness of solutions.} \label{app:a}
Following \cite{C1}, let us show Lemma \ref{lem:uni}.
We define some notations
\begin{align*}
	&{}f(t) :=f(\vec{u}(t)) := M(\vec{u}(t)) %\label{app:1}
%= \Lebn{u_1(t)}{2}^2 + \frac12 \Lebn{u_2(t)}{2}^2 
\end{align*}
and $[x]^+ = \max(x, 0)$ for any $x \in \R$.
\begin{proposition} \label{app:pro1}
Let $\vec{u} \in C([0,\I), \m{H}^1)$ be a solution of \eqref{nlkg}. Let $b = \min(1, \ka)$. Then it holds that
\begin{align}
	\frac{d}{dt}\(\left[ b^2 f(t) - 6 E(\vec{u}(0), \pa_t \vec{u}(0)) \right]^+ \) \le 0, \label{app:2} \\
	f(t) \le \sup \( f(0), \frac6{b^2} E(\vec{u}(0), \pa_t \vec{u}(0))\) \label{app:3}
\end{align}
for any $t \in [0,\I)$. Moreover, 
\begin{align}
	E(\vec{u}(0), \pa_t \vec{u}(0)) \ge 0 \label{app:4}
\end{align}
is valid.
\end{proposition}
\begin{remark}
\eqref{app:3} implies \eqref{lem:01}.
\end{remark}
\begin{proof}
Set
\[
	g(t) = f(t) - \frac{6}{b^2} E(\vec{u}(0), \pa_t \vec{u}(0)).
\]
Let us show by the contradiction. Namely, assume that there exists $t_1 \in [0,\I)$ such that $$\frac{d}{dt}\( [ b^2 f(t_1) - 6 E(\vec{u}(0), \pa_t \vec{u}(0))]^+ \) >0.$$ Hence we have $g'(t_1)>0$ and $g(t_1)>0$. Further one sees that
\begin{align}
	\begin{aligned}
	g''(t) =f''(t)={}& 5M(\pa_t \vec{u}(t)) - 6 E(\vec{u}(0), \pa_t \vec{u}(0)) \\
	&{}+ \Lebn{u_1(t)}{2}^2 + \frac{\ka^2}{2} \Lebn{u_2(t)}{2}^2 + L(\vec{u}(t)),
	\end{aligned}
	\label{app:5}
\end{align}
which implies
\begin{align*}
	g''(t) \ge b^2 M(\vec{u}(t)) - 6 E(\vec{u}(0), \pa_t \vec{u}(0)) \ge b^2 g(t)
\end{align*}
for any $t \in [0,\I)$. Therefore, $g$ is a convex increasing function on $[t_1,\I)$ with $\lim_{t \to \I}g(t) = \I$. This gives us $g(t) \ge 0$ for any $t \ge t_1$. By using this fact, we estimate
\begin{align*}
	f''(t) = {}& 5M(\pa_t \vec{u}(t)) - 6 E(\vec{u}(0), \pa_t \vec{u}(0)) \\
	&{}+ \Lebn{u_1(t)}{2}^2 + \frac{\ka^2}{2} \Lebn{u_2(t)}{2}^2 + L(\vec{u}(t)) \\
	\ge{}& 5M(\pa_t \vec{u}(t)) + b^2 M(\vec{u}(t)) - 6 E(\vec{u}(0), \pa_t \vec{u}(0)) \\
	\ge{}& 5M(\pa_t \vec{u}(t)).
\end{align*}
This tells us that
\begin{align*}
	(f'(t))^2 ={}& 4\( \re \( u_1, \pa_t u_1\)_{L^2} \)^2 + 4 \re \( u_1, \pa_t u_1\)_{L^2} \re \( u_2, \pa_t u_2 \)_{L^2} \\
	&{}+ \( \re \( u_2, \pa_t u_2\)_{L^2} \)^2 \\
	\le{}& 4 \Lebn{u_1}{2}^2 \Lebn{\pa_t u_1}{2}^2 + 4 \Lebn{u_1}{2}\Lebn{\pa_t u_1}{2} \Lebn{u_2}{2} \Lebn{\pa_t u_2}{2} \\
	&{}+ \Lebn{u_2}{2}^2 \Lebn{\pa_t u_2}{2}^2 \\
	\le{}& 4 M(\vec{u}(t)) M(\pa_t \vec{u}(t)) \le \frac45 f(t) f''(t),
\end{align*}
which yields
\[
	\frac54 (f'(t))^2 \le f(t)f''(t)
\]
for any $t \ge t_1$. From the above, we see that
\[
	\((f(t))^{-\frac14}\)'' = -\frac14 f^{-\frac94} \( -\frac54(f')^2 + f f'' \) \le 0.
\]
Hence, $(f(t))^{-\frac14}$ is concave on $[t_1, \I)$. On the other hand, since $g$ is a increasing function on $[t_1, \I)$, $f$ is also increasing on $[t_1, \I)$. This implies $(f(t))^{-\frac14} \to 0$ as $t \to \I$, which contradicts that $(f(t))^{-\frac14}$ is concave on $[t_1, \I)$. Therefore, we have \eqref{app:2}. Furthermore, by \eqref{app:2}, it holds that
\[
	b^2 f(t) -6 E(\vec{u}(0), \pa_t \vec{u}(0)) \le \left[ b^2 f(0) -6 E(\vec{u}(0), \pa_t \vec{u}(0)) \right]^+
\]
for any $t \in [0,\I)$. This yields \eqref{app:3}. We shall show \eqref{app:4}. If $E(\vec{u}(0), \pa_t \vec{u}(0))<0$, then we have
\[
	f''(t) \ge -6E(\vec{u}(0), \pa_t \vec{u}(0)) =:\a >0.
\]
This implies $f'(t) \ge \a t + f'(0)$ for any $t \ge 0$. Hence, there exists $t_2>0$ such that $f'(t) >0$ for any $t \in [t_2, \I)$, which implies $f(t) \to \I$ as $t \to \I$. This contradicts \eqref{app:2}. The proof is completed. 
\end{proof}

\begin{lemma}
Let $\vec{u} \in C([0,\I), \m{H}^1)$ be a solution of \eqref{nlkg}. Let $b = \min(1, \ka)$. Then it holds that
\begin{align}
	f'(t) \le{}& \frac{6}{\sqrt{5}b} E(\vec{u}(0), \pa_t \vec{u}(0)), \label{app:6} \\
	f'(t) \ge{}& \inf \( f'(0), -\frac{6}{\sqrt{5}b} E(\vec{u}(0), \pa_t \vec{u}(0)) \) \label{app:7}
\end{align}
for any $t \in [0,\I)$. 
\end{lemma}
\begin{proof}
Set
\[
	h(t) = f'(t) - \frac{6}{\sqrt{5}b} E(\vec{u}(0), \pa_t \vec{u}(0)).
\]
A use of the Young inequality $2xy \le \e x^2 + \e^{-1} y^2$ with $\e = \frac{b}{\sqrt{5}}$ gives us 
\begin{align*}
	|f'(t)| \le{}& 2\left|\( u_1, \pa_t u_1\)_{L^2} \right| + \left|\( u_2, \pa_t u_2\)_{L^2} \right| \\
	\le{}& 2 \Lebn{u_1}{2} \Lebn{\pa_t u_1}{2} + \Lebn{u_2}{2} \Lebn{\pa_t u_2}{2} \\
	\le{}& \frac{b}{\sqrt{5}}\Lebn{u_1}{2}^2 + \frac{\sqrt{5}}{b} \Lebn{\pa_t u_1}{2}^2 + \frac12 \( \frac{b}{\sqrt{5}} \Lebn{u_2}{2}^2 + \frac{\sqrt{5}}{b} \Lebn{\pa_t u_2}{2}^2 \).
\end{align*}
This yields
\begin{align}
	\sqrt{5}b |f'(t)| \le{}&  \Lebn{u_1(t)}{2}^2 + \frac{\ka^2}{2} \Lebn{u_2(t)}{2}^2 + 5 M(\pa_t \vec{u}(t)).
	\label{app:8}
\end{align}
By means of \eqref{app:5} and \eqref{app:8}, we have
\begin{align*}
	h'(t) = f''(t) \ge \sqrt{5}b |f'(t)| - 6E(\vec{u}(0), \pa_t \vec{u}(0)) \ge \sqrt{5}b h(t)
\end{align*}
for any $t \in [0,\I)$. This tells us that
\begin{align}
	h(t) \ge h(t_0)e^{\sqrt{5}b(t-t_0)} \label{app:9}
\end{align}
for any $t_0$, $t \in [0,\I)$ with $t_0 \le t$. Let us suppose that there exists $t_1 \in [0,\I)$ such $h(t_1)>0$. 
Then we see from \eqref{app:9} that
\[
	f'(t) \ge \frac{6}{\sqrt{5}b} E(\vec{u}(0), \pa_t \vec{u}(0)) + h(t_1)e^{\sqrt{5}b(t-t_1)}>0
\]
for any $t>t_1$, which yields $f(t) \to \I$ as $t \to \I$. This contradicts \eqref{app:3}. 
We then complete the proof of \eqref{app:6}. Let us prove \eqref{app:7}. Put
\[
	k(t) = -f'(t) - \frac{6}{\sqrt{5}b}E(\vec{u}(0), \pa_t \vec{u}(0)).
\]
By Combining \eqref{app:5} with \eqref{app:8}, we have
\begin{align*}
	-k'(t) = f''(t) \ge{}& F(t) -6 E(\vec{u}(0), \pa_t \vec{u}(0)) \\
	\ge{}& \sqrt{5}b |f'(t)| - 6 E(\vec{u}(0), \pa_t \vec{u}(0)) \\
	\ge{}& - \sqrt{5}b f'(t) - 6 E(\vec{u}(0), \pa_t \vec{u}(0)) \\
	={}& \sqrt{5}b k(t).
\end{align*}
This gives us $k(t) \le k(0) e^{-\sqrt{5}b t}$ for any $t \in [0,\I)$. Thus, it holds that $k(t) \le \sup(k(0), 0)$, which implies \eqref{app:7}. We complete the proof. 
\end{proof}

\begin{lemma} \label{app:lem1}
Let $\vec{u} \in C([0,\I), \m{H}^1)$ be a solution of \eqref{nlkg}. Then the estimate 
\begin{align*}
	\int_{t}^{t+\tau} F(s)\, ds \le{}& 6E(\vec{u}(0), \pa_t \vec{u}(0)) \tau + \frac{12}{\sqrt{5}b}E(\vec{u}(0), \pa_t \vec{u}(0)) \\
	&{}+ 2\left|\(u_1(0), \pa_t u_1(0) \)_{L^2}\right| + \left|\(u_2(0), \pa_t u_2(0) \)_{L^2}\right|
\end{align*}
for any $t \ge 0$ and all $\tau >0$, where
\begin{align*}
	F(t) ={}& F(\vec{u}(s), \pa_t \vec{u}(s)) \\
	={}& 5M(\pa_t \vec{u}(t)) + \Lebn{u_1(t)}{2}^2 + \frac{\ka^2}{2} \Lebn{u_2(t)}{2}^2 + L(\vec{u}(t)).
\end{align*}
\end{lemma}
\begin{remark}
Lemma \ref{app:lem1} gives us \eqref{lem:02}.
\end{remark}
\begin{proof}
Integrating the both side of \eqref{app:5} for $t$, we have
\begin{align*}
	\int_{t}^{t+\tau} f''(s)\, ds = \int_t^{t+\tau} F(s)\, ds - 6E(\vec{u}(0), \pa_t \vec{u}(0)) \tau. 
\end{align*}
This implies that
\[
	\int_t^{t+\tau} F(s)\, ds = 6E(\vec{u}(0), \pa_t \vec{u}(0)) \tau + f'(t+\tau) -f'(t).
\]
Hence, the desired estimate follows from \eqref{app:6} and \eqref{app:7}.
\end{proof}

\subsection*{Acknowledgments} 
The part of this work was done while the author
was visiting at Department of Mathematics at the University of California,
Santa Barbara whose hospitality he gratefully acknowledges.
The author is very grateful to Dr. Masahiro Ikeda for the helpful advice and valuable suggestions to consider this problem.
The author would also like to express deep gratitude to Professor Masahito Ohta for introducing him the paper \cite{BJM1}. 
The author was supported by JSPS KAKENHI Grant Numbers 19K14580 and the Overseas Research Fellowship Program by National Institute of Technology.

%\bibliography{Instability_NLKGS_2}

\begin{bibdiv}
\begin{biblist}

\bib{AF1}{book}{
      author={Adams, Robert~A.},
      author={Fournier, John J.~F.},
       title={Sobolev spaces},
     edition={Second},
      series={Pure and Applied Mathematics (Amsterdam)},
   publisher={Elsevier/Academic Press, Amsterdam},
        date={2003},
      volume={140},
        ISBN={0-12-044143-8},
      review={\MR{2424078}},
}

\bib{BC}{article}{
      author={Berestycki, Henri},
      author={Cazenave, Thierry},
       title={Instabilit\'{e} des \'{e}tats stationnaires dans les
  \'{e}quations de {S}chr\"{o}dinger et de {K}lein-{G}ordon non lin\'{e}aires},
        date={1981},
        ISSN={0249-6321},
     journal={C. R. Acad. Sci. Paris S\'{e}r. I Math.},
      volume={293},
      number={9},
       pages={489\ndash 492},
      review={\MR{646873}},
}

\bib{BL1}{article}{
      author={Brezis, Ha\"{\i}m},
      author={Lieb, Elliott~H.},
       title={Minimum action solutions of some vector field equations},
        date={1984},
        ISSN={0010-3616},
     journal={Comm. Math. Phys.},
      volume={96},
      number={1},
       pages={97\ndash 113},
         url={http://projecteuclid.org/euclid.cmp/1103941720},
      review={\MR{765961}},
}

\bib{BJM1}{article}{
      author={Byeon, Jaeyoung},
      author={Jeanjean, Louis},
      author={Mari\c{s}, Mihai},
       title={Symmetry and monotonicity of least energy solutions},
        date={2009},
        ISSN={0944-2669},
     journal={Calc. Var. Partial Differential Equations},
      volume={36},
      number={4},
       pages={481\ndash 492},
         url={https://doi.org/10.1007/s00526-009-0238-1},
      review={\MR{2558325}},
}

\bib{C1}{article}{
      author={Cazenave, Thierry},
       title={Uniform estimates for solutions of nonlinear {K}lein-{G}ordon
  equations},
        date={1985},
        ISSN={0022-1236},
     journal={J. Funct. Anal.},
      volume={60},
      number={1},
       pages={36\ndash 55},
         url={https://doi.org/10.1016/0022-1236(85)90057-6},
      review={\MR{780103}},
}

\bib{C2}{book}{
      author={Cazenave, Thierry},
       title={Semilinear {S}chr\"odinger equations},
      series={Courant Lecture Notes in Mathematics},
   publisher={New York University, Courant Institute of Mathematical Sciences,
  New York; American Mathematical Society, Providence, RI},
        date={2003},
      volume={10},
        ISBN={0-8218-3399-5},
         url={https://doi.org/10.1090/cln/010},
      review={\MR{2002047}},
}

\bib{CP}{article}{
      author={Comech, Andrew},
      author={Pelinovsky, Dmitry},
       title={Purely nonlinear instability of standing waves with minimal
  energy},
        date={2003},
        ISSN={0010-3640},
     journal={Comm. Pure Appl. Math.},
      volume={56},
      number={11},
       pages={1565\ndash 1607},
         url={https://doi.org/10.1002/cpa.10104},
      review={\MR{1995870}},
}

\bib{V1}{article}{
      author={Duong~Dinh, Van},
       title={Existence, stability of standing waves and the characterization
  of finite time blow-up solutions for a system {NLS} with quadratic
  interaction},
     journal={preprint},
      eprint={arXiv:1809.09643},
}

\bib{V2}{article}{
      author={Duong~Dinh, Van},
       title={Strong instability of standing waves for a system nls with
  quadratic interaction},
     journal={preprint},
      eprint={arXiv:1810.00676},
}

\bib{DG1}{article}{
      author={Garrisi, Daniele},
       title={On the orbital stability of standing-wave solutions to a coupled
  non-linear {K}lein-{G}ordon equation},
        date={2012},
        ISSN={1536-1365},
     journal={Adv. Nonlinear Stud.},
      volume={12},
      number={3},
       pages={639\ndash 658},
         url={https://doi.org/10.1515/ans-2012-0311},
      review={\MR{2976057}},
}

\bib{GSS2}{article}{
      author={Grillakis, Manoussos},
      author={Shatah, Jalal},
      author={Strauss, Walter},
       title={Stability theory of solitary waves in the presence of symmetry.
  {I}},
        date={1987},
        ISSN={0022-1236},
     journal={J. Funct. Anal.},
      volume={74},
      number={1},
       pages={160\ndash 197},
         url={https://doi.org/10.1016/0022-1236(87)90044-9},
      review={\MR{901236}},
}

\bib{GSS1}{article}{
      author={Grillakis, Manoussos},
      author={Shatah, Jalal},
      author={Strauss, Walter},
       title={Stability theory of solitary waves in the presence of symmetry.
  {II}},
        date={1990},
        ISSN={0022-1236},
     journal={J. Funct. Anal.},
      volume={94},
      number={2},
       pages={308\ndash 348},
         url={https://doi.org/10.1016/0022-1236(90)90016-E},
      review={\MR{1081647}},
}

\bib{Ha1}{article}{
      author={Hamano, Masaru},
       title={Global dynamics below the ground state for the quadratic
  schodinger system in 5d},
     journal={preprint},
      eprint={arXiv:1805.12245},
}

\bib{HIN1}{article}{
      author={Hayashi, Nakao},
      author={Ikeda, Masahiro},
      author={Naumkin, Pavel~I.},
       title={Wave operator for the system of the {D}irac-{K}lein-{G}ordon
  equations},
        date={2011},
        ISSN={0170-4214},
     journal={Math. Methods Appl. Sci.},
      volume={34},
      number={8},
       pages={896\ndash 910},
         url={https://doi.org/10.1002/mma.1409},
      review={\MR{2828738}},
}

\bib{HOT}{article}{
      author={Hayashi, Nakao},
      author={Ozawa, Tohru},
      author={Tanaka, Kazunaga},
       title={On a system of nonlinear {S}chr\"odinger equations with quadratic
  interaction},
        date={2013},
        ISSN={0294-1449},
     journal={Ann. Inst. H. Poincar\'e Anal. Non Lin\'eaire},
      volume={30},
      number={4},
       pages={661\ndash 690},
         url={https://doi.org/10.1016/j.anihpc.2012.10.007},
      review={\MR{3082479}},
}

\bib{L}{article}{
      author={Lieb, Elliott~H.},
       title={On the lowest eigenvalue of the {L}aplacian for the intersection
  of two domains},
        date={1983},
        ISSN={0020-9910},
     journal={Invent. Math.},
      volume={74},
      number={3},
       pages={441\ndash 448},
         url={https://doi.org/10.1007/BF01394245},
      review={\MR{724014}},
}

\bib{LL1}{book}{
      author={Lieb, Elliott~H.},
      author={Loss, Michael},
       title={Analysis},
     edition={Second},
      series={Graduate Studies in Mathematics},
   publisher={American Mathematical Society, Providence, RI},
        date={2001},
      volume={14},
        ISBN={0-8218-2783-9},
         url={https://doi.org/10.1090/gsm/014},
      review={\MR{1817225}},
}

\bib{LOT}{article}{
      author={Liu, Yue},
      author={Ohta, Masahito},
      author={Todorova, Grozdena},
       title={Instabilit\'{e} forte d'ondes solitaires pour des \'{e}quations
  de {K}lein-{G}ordon non lin\'{e}aires et des \'{e}quations
  g\'{e}n\'{e}ralis\'{e}es de {B}oussinesq},
        date={2007},
        ISSN={0294-1449},
     journal={Ann. Inst. H. Poincar\'{e} Anal. Non Lin\'{e}aire},
      volume={24},
      number={4},
       pages={539\ndash 548},
         url={https://doi.org/10.1016/j.anihpc.2006.03.005},
      review={\MR{2334991}},
}

\bib{M1}{article}{
      author={Maeda, Masaya},
       title={Stability of bound states of {H}amiltonian {PDE}s in the
  degenerate cases},
        date={2012},
        ISSN={0022-1236},
     journal={J. Funct. Anal.},
      volume={263},
      number={2},
       pages={511\ndash 528},
         url={https://doi.org/10.1016/j.jfa.2012.04.006},
      review={\MR{2923422}},
}

\bib{MZ}{article}{
      author={Merle, Frank},
      author={Zaag, Hatem},
       title={Determination of the blow-up rate for the semilinear wave
  equation},
        date={2003},
        ISSN={0002-9327},
     journal={Amer. J. Math.},
      volume={125},
      number={5},
       pages={1147\ndash 1164},
  url={http://muse.jhu.edu/journals/american_journal_of_mathematics/v125/125.5merle.pdf},
      review={\MR{2004432}},
}

\bib{Na}{article}{
      author={Nawa, Hayato},
       title={Asymptotic profiles of blow-up solutions of the nonlinear
  {S}chr\"{o}dinger equation with critical power nonlinearity},
        date={1994},
        ISSN={0025-5645},
     journal={J. Math. Soc. Japan},
      volume={46},
      number={4},
       pages={557\ndash 586},
         url={https://doi.org/10.2969/jmsj/04640557},
      review={\MR{1291107}},
}

\bib{OgT1}{article}{
      author={Ogawa, Takayoshi},
      author={Tsutsumi, Yoshio},
       title={Blow-up of {$H^1$} solution for the nonlinear {S}chr\"{o}dinger
  equation},
        date={1991},
        ISSN={0022-0396},
     journal={J. Differential Equations},
      volume={92},
      number={2},
       pages={317\ndash 330},
         url={https://doi.org/10.1016/0022-0396(91)90052-B},
      review={\MR{1120908}},
}

\bib{OgT2}{article}{
      author={Ogawa, Takayoshi},
      author={Tsutsumi, Yoshio},
       title={Blow-up of {$H^1$} solutions for the one-dimensional nonlinear
  {S}chr\"{o}dinger equation with critical power nonlinearity},
        date={1991},
        ISSN={0002-9939},
     journal={Proc. Amer. Math. Soc.},
      volume={111},
      number={2},
       pages={487\ndash 496},
         url={https://doi.org/10.2307/2048340},
      review={\MR{1045145}},
}

\bib{OT2}{article}{
      author={Ohta, Masahito},
      author={Todorova, Grozdena},
       title={Strong instability of standing waves for nonlinear
  {K}lein-{G}ordon equations},
        date={2005},
        ISSN={1078-0947},
     journal={Discrete Contin. Dyn. Syst.},
      volume={12},
      number={2},
       pages={315\ndash 322},
         url={https://doi.org/10.1137/050643015},
      review={\MR{2122169}},
}

\bib{OT}{article}{
      author={Ohta, Masahito},
      author={Todorova, Grozdena},
       title={Strong instability of standing waves for the nonlinear
  {K}lein-{G}ordon equation and the {K}lein-{G}ordon-{Z}akharov system},
        date={2007},
        ISSN={0036-1410},
     journal={SIAM J. Math. Anal.},
      volume={38},
      number={6},
       pages={1912\ndash 1931},
         url={https://doi.org/10.1137/050643015},
      review={\MR{2299435}},
}

\bib{PS1}{article}{
      author={Payne, L.~E.},
      author={Sattinger, D.~H.},
       title={Saddle points and instability of nonlinear hyperbolic equations},
        date={1975},
        ISSN={0021-2172},
     journal={Israel J. Math.},
      volume={22},
      number={3-4},
       pages={273\ndash 303},
         url={https://doi.org/10.1007/BF02761595},
      review={\MR{0402291}},
}

\bib{P}{article}{
      author={Pecher, Hartmut},
       title={Nonlinear small data scattering for the wave and {K}lein-{G}ordon
  equation},
        date={1984},
        ISSN={0025-5874},
     journal={Math. Z.},
      volume={185},
      number={2},
       pages={261\ndash 270},
         url={https://doi.org/10.1007/BF01181697},
      review={\MR{731347}},
}

\bib{Sh}{article}{
      author={Shatah, Jalal},
       title={Stable standing waves of nonlinear {K}lein-{G}ordon equations},
        date={1983},
        ISSN={0010-3616},
     journal={Comm. Math. Phys.},
      volume={91},
      number={3},
       pages={313\ndash 327},
         url={http://projecteuclid.org/euclid.cmp/1103940612},
      review={\MR{723756}},
}

\bib{Sh2}{article}{
      author={Shatah, Jalal},
       title={Unstable ground state of nonlinear {K}lein-{G}ordon equations},
        date={1985},
        ISSN={0002-9947},
     journal={Trans. Amer. Math. Soc.},
      volume={290},
      number={2},
       pages={701\ndash 710},
         url={https://doi.org/10.2307/2000308},
      review={\MR{792821}},
}

\bib{ShS}{article}{
      author={Shatah, Jalal},
      author={Strauss, Walter},
       title={Instability of nonlinear bound states},
        date={1985},
        ISSN={0010-3616},
     journal={Comm. Math. Phys.},
      volume={100},
      number={2},
       pages={173\ndash 190},
         url={http://projecteuclid.org/euclid.cmp/1103943442},
      review={\MR{804458}},
}

\bib{St}{article}{
      author={Strauss, Walter~A.},
       title={Existence of solitary waves in higher dimensions},
        date={1977},
        ISSN={0010-3616},
     journal={Comm. Math. Phys.},
      volume={55},
      number={2},
       pages={149\ndash 162},
         url={http://projecteuclid.org/euclid.cmp/1103900983},
      review={\MR{0454365}},
}

\bib{YT1}{article}{
      author={Tsutsumi, Yoshio},
       title={Stability of constant equilibrium for the {M}axwell-{H}iggs
  equations},
        date={2003},
        ISSN={0532-8721},
     journal={Funkcial. Ekvac.},
      volume={46},
      number={1},
       pages={41\ndash 62},
         url={https://doi.org/10.1619/fesi.46.41},
      review={\MR{1996293}},
}

\bib{BW1}{article}{
      author={Wang, Baoxiang},
       title={On existence and scattering for critical and subcritical
  nonlinear {K}lein-{G}ordon equations in {$H^s$}},
        date={1998},
        ISSN={0362-546X},
     journal={Nonlinear Anal.},
      volume={31},
      number={5-6},
       pages={573\ndash 587},
         url={https://doi.org/10.1016/S0362-546X(97)00424-0},
      review={\MR{1487847}},
}

\bib{Wu}{article}{
      author={Wu, Yifei},
       title={Instability of the standing waves for the nonlinear klein-gordon
  equations in one dimension},
     journal={preprint},
      eprint={arXiv:1705.04216},
}

\bib{ZGG1}{article}{
      author={Zhang, Jian},
      author={Gan, Zai-hui},
      author={Guo, Bo-ling},
       title={Stability of the standing waves for a class of coupled nonlinear
  {K}lein-{G}ordon equations},
        date={2010},
        ISSN={0168-9673},
     journal={Acta Math. Appl. Sin. Engl. Ser.},
      volume={26},
      number={3},
       pages={427\ndash 442},
         url={https://doi.org/10.1007/s10255-010-0008-z},
      review={\MR{2657700}},
}

\end{biblist}
\end{bibdiv}


\end{document}